%% file: WadgeScottv2.tex
\documentclass[10pt,a4paper]{article}

\usepackage{comment}
\input{./macros/packages}
\input{./macros/macros}

\input{./macros/title}

\begin{document}

\maketitle

\begin{abstract}
We prove that the Wadge order on the Borel subsets of the Scott domain is not a well-quasi-order, and that this feature even occurs among the sets of Borel rank at most 2. For this purpose, a specific class of countable 2-colored posets $\mathbb{P}_{\lay}$ equipped with the order induced by homomorphisms is embedded into the Wadge order on the $\bdelta^0_2$-degrees of the Scott domain. We then show that $\mathbb{P}_{\lay}$ both admits infinite strictly decreasing chains and infinite antichains with respect to this notion of comparison, which therefore transfers to the Wadge order on the $\bdelta^0_2$-degrees of the Scott domain. \\
\emph{Keywords}: quasi-Polish space, Scott domain, Wadge reducibility, well-quasi-order. \\
\emph{Mathematics Subject Classification}: 03B70, 03D30, 03E15, 54H05, 68Q15.
\end{abstract}

\vspace{0.5cm}

%\tableofcontents

\normalsize

With the exception of Section \ref{sectionduparc}, all the results presented in this article -- including the main ones -- are due to the sole second author.

%\newpage

\section{Introduction}

The \emph{Wadge order} $\leq_w$ -- named after Wadge \cite{Wadge1984} -- on the subsets of a topological space $X$ is the quasi-order induced by reductions via continuous functions. More precisely, if $A,B\subseteq X,$ then $A\leq_w B$ if there exists a continuous function $f:X\to X$ such that $f^{-1}[B]=A,$ i.e., $x\in A\Leftrightarrow f(x)\in B$ for all $x\in X.$ The Wadge order measures the topological complexity of the subsets of $X$. Indeed, $A\leq_w B$ means that the membership problem for $A$ can be reduced, via some continuous function, to the membership problem for $B;$ or, in other words, $A$ is topologically less complicated than $B.$

The Wadge order is a refinement of both the classical Borel and Hausdorff-Kuratowski difference hierarchies since when $B$ is located strictly higher than $A$ in one of these hierarchies, then $A\leq_w B$ holds. Over the last 50 years, this quasi-order has been extensively studied in the context of \emph{Polish spaces} -- i.e., the separable completely metrizable spaces \cite{Andretta2012,Andretta2007,Duparc2001,Ikegami2010bis,Kechris1995,Kihara2017,Louveau2012,Louveau2012bis,Schlicht2018,Vanwesep2012,Wadge1984,Wadge2012}.

Over the last decades, some slightly different classes of topological spaces rose interest for their involvement in computer science \cite{Gierz2003,Goubaultlarrecq2013, Scott1976, Selivanov2005bis, Selivanov2006, Weihrauch2000}. This has been the case, in particular, of non-metrizable -- hence non-Polish -- spaces occurring as domains of the semantic of programming languages. Building on a prior work of Selivanov -- that extensively studied a generalized version of the Borel hierarchy to non-metrizable spaces \cite{Selivanov2005bis,Selivanov2006} -- de Brecht introduced in \cite{Debrecht2012} the class of \emph{quasi-Polish spaces} -- i.e., the second countable quasi-metrizable spaces, where a quasi-metric is a metric whose symmetry condition has been dropped. In particular, de Brecht proved that some of the major results of descriptive set theory extend to quasi-Polish spaces (see Theorems 19, 23, 58 and 70 in \cite{Debrecht2012}). He also exhibited the \emph{Scott domain}\footnote{The Scott domain was first introduced by Scott as a denotational semantic for the $\lambda$-calculus \cite{Scott1976}.} $\scott$ as a universal quasi-Polish space. More precisely, de Brecht proved that the quasi-Polish spaces are -- up to homeomorphism -- exactly the $\bpi^0_2$-subsets of $\scott$ (Theorem 24 in \cite{Debrecht2012}), where $\scott$ is the power set of the integers equipped with the topology where a basic open set is composed of all the sets that contain a fixed finite subset of the integers.

More results by de Brecht suggest that a reasonable {descriptive set theory} still holds in the quasi-Polish setting. Unfortunately, very few is known about the Wadge order in this context. To the contrary, the Polish spaces $X$ whose Wadge order on the Borel subsets is well-founded and contains no infinite antichain -- or in other words, $\leq_w$ is a well-quasi-order on the Borel subsets of $X$ -- were recently characterized in \cite{Schlicht2018} as the \emph{zero-dimensional} ones -- i.e., Polish spaces admitting a clopen basis. Whether this result generalizes to quasi-Polish spaces remains open. In a first attempt to tackle this question, we propose to study the Wadge order on the subsets of the Scott domain $\scott.$

Several results have already been obtained by Selivanov who proved the existence of $\leq_w$-antichains of size 4 for $\scott$, as well as the existence of $\leq_w$-minimal sets at each level of the difference hierarchy of open sets \cite{Selivanov2005bis}; and by Becher and Grigorieff who exhibited, for each infinite level $\alpha$ of the difference hierarchy of open sets, some strictly $\leq_w$-increasing chains of sets of length $\alpha$, and also described the $\leq_w$-maximal sets for each such level for a large number of quasi-Polish spaces including $\scott$ \cite{Becher2015bis}. In this article, we show both that the Wadge order on the subsets of $\scott$ is ill-founded and that it admits infinite antichains. Moreover, we show that these properties occur already within the differences of $\omega$ open sets, i.e., at a very low level of topological complexity:

\setcounter{theorembis}{\getrefnumber{illfounded}-1}

\begin{theorembis}
$\big(D_\omega(\bsigma^0_1)(\Pw),\leq_w\!\!\big)$ is ill-founded.
\end{theorembis}

\setcounter{theorembis}{\getrefnumber{antichain}-1}

\begin{theorembis}
$\big(D_\omega(\bsigma^0_1)(\Pw),\leq_w\!\!\big)$ has infinite antichains.
\end{theorembis}

These results are obtained through a generalization of a construction introduced by Selivanov in \cite{Selivanov2005bis}. More precisely, we define an order-embedding from a class of 2-colored countable posets $\mathbb{P}_{\lay}$ (Definition \ref{defplay}) endowed with the usual notion of comparison by homomorphisms into the Wadge order on the $\bdelta^0_2$-degrees of $\scott$, where a degree is an equivalence class induced by $\leq_w$:

\setcounter{theorembis}{\getrefnumber{injectivehomomorphism}-1}

\begin{theorembis}
There exists an order-embedding:
\begin{align*}
\sfrac{({\mathbb{P}_{\reg}},\preccurlyeq_c)}{\equiv_c} &\to \sfrac{(\bdelta^0_2(\Pw),\leq_w)}{\equiv_w}.
\end{align*}
\end{theorembis}

Different approaches have already been considered for tackling the problem of classifying subsets of non-Polish spaces according to their topological complexity. For instance, Pequignot studied the quasi-order obtained from reductions via admissible representations \cite{Pequignot2015}, and Motto Ros, Schlicht and Selivanov investigated the quasi-order obtained from classes of reductions that are larger than the continuous ones \cite{Mottoros2015}.

The article is organized as follows. We fix notations and general definitions in Section 2, where we also recall results such as the characterizations of some of the topological classes obtained by Selivanov in \cite{Selivanov2005bis}. In Section 3, we define the class of posets $\mathbb{P}_{\lay}$ (Definition \ref{defplay}) that we embed into the Wadge order of $\scott$ (Theorem \ref{injectivehomomorphism}) in Section 4. This order-embedding is the main construction of this article. A game characterization of reductions between 2-colored posets is introduced in Section 5 (Definition \ref{posetgame}) in order to show, in Section 6 and 7, that the Wadge order of $\scott$ is ill-founded (Theorem \ref{illfounded}) and that it has infinite antichains (Theorem \ref{antichain}). We conclude in Section 8 with open questions.

\section{Preliminaries}

%We first introduce some notations.

\subsection{General notations}

As usual, we denote by $\w$ or $\N$ the set of all integers and by $\aleph_0$ its cardinality. We also write $\w^+$ for $\w\setminus\{0\}$ and $\w_1$ for the first uncountable ordinal. We use the letters $i,j,k,l,m,n$ for integers and $\alpha,\beta,\gamma$ for arbitrary ordinals. Since every ordinal is regarded as the set of its predecessors, if $n\in\w,$ the notation $x\cap n$ stands for $x\cap\{0,1,\dots,n-1\}.$  

Given any sets $X,Y$, if $f:X\to Y$ is a function, $A\subseteq X,$ and $B\subseteq Y$, then we write $f[A]=\{f(x)\mid x\in A\}$ and $f^{-1}[B]=\{x\mid f(x)\in B\}.$ If $f$ is injective, we write $f^{-1}(y)$ for the unique element $x\in X$ such that $f(x)=y$.

An $X$-sequence -- or simply a sequence -- is a function $s:\alpha\to X$ -- denoted by $(s_\beta)_{\beta<\alpha}$ -- from some ordinal $\alpha$ called the length of the sequence to $X$. In this article, we will mainly consider sequences such that $\alpha\in \w+1= \omega\cup\{\omega\}$. We use the letters $s,t$ to denote sequences. The only sequence of length $0$ -- the empty sequence -- is denoted by $\varnothing$. If $s,t$ are sequences, then $t$ is a prefix of $s,$ written $t\sqsubseteq s$, if $\lh(t)\leq \lh(s)$ and $s_k=t_k$ for all $k<\lh(t)$. If $t\sqsubseteq s$ but $s\not\sqsubseteq t$, we write $t\sqsubset s.$ If $s,t$ are $X$-sequences, the concatenation of $s$ and $t$ is defined by $s^\frown t=(s_0,\dots,s_{\lh(s)-1},t_0,\dots,t_{\lh(t)-1})$. The set of all $X$-sequences of finite length is denoted by $X^{<\w}.$

A tree $T\subseteq X^{<\w}$ is a set of finite $X$-sequences closed under the prefix relation\footnote{If $t\in T$ and $s\sqsubseteq t$, then $s\in T$.}. It is well-founded if it has no infinite branch\footnote{An infinite branch is a function $f:\w\to T$ such that, if $n<m,$ then $f(n)\sqsubset f(m)$.}, in which case the rank of any $t\in T$ is (well-)defined by $\sqsupseteq$-induction: $\rk_T(t)=0$ if $t$ is $\sqsubseteq$-maximal and $\rk_T(t)=\sup\{\rk_T(s)+1\mid t\sqsubset s\}$ otherwise. The rank $\rk(T)$ of a non-empty well-founded tree $T$ is the ordinal $\rk_T(\varnothing)$.

\subsection{Order-theoretic notations}

A quasi-order on a set $Q$ is any reflexive and transitive relation\footnote{A binary relation $\leq_q$ on $Q$ is reflexive if, for all $q\in Q$, $(q,q)\in\ \leq_q$, and transitive if, for any $q_0,q_1,q_2\in Q,$ $(q_0,q_1),(q_1,q_2)\in\ \leq_q$ implies $(q_0,q_2)\in\ \leq_q$.} $\leq_q\ \subseteq Q\times Q$. Whenever $\leq_q$ is clear from the context, we write $Q$ for the couple $(Q,\leq_q)$. We will use the letters $P,Q$ for quasi-orders and $p\in P,q\in Q$ for their elements. As usual, $q_0\leq_q q_1$ stands for $(q_0,q_1)\in\ \leq_q$, and $q_0<_qq_1$ for $q_0\leq_q q_1$ but $q_1\not\leq_q q_0$. If $q_0\nleq_q q_1$ and $q_1\nleq_q q_0,$ then $q_0$ and $q_1$ are said to be incomparable which is denoted by $q_0\perp_q q_1.$ If $Q$ is a quasi-order and $P\subseteq Q,$ then $P$ equipped with the induced relation is a quasi-order. An infinite antichain in $Q$ is a sequence $(q_n)_{n<\omega}$ of pairwise incomparable elements, and a strictly $\leq_q$-increasing (resp. strictly $\leq_q$-decreasing) sequence is a sequence $(q_n)_{n<\omega}$ such that $q_{n}<_q q_{n+1}$ (resp. $q_{n+1}<_q q_{n}$) for all $n\in\omega.$ A well-quasi-order is a quasi-order $Q$ that has no infinite antichain and no strictly $\leq_q$-decreasing sequence. We denote by $\Pred(q)=\{q'\in Q\mid q'\leq_q q\}$ the set of predecessors of $q\in Q,$ and by $\Pred_{\imm}(q)=\{q'\in Q\mid (q'<_q q) \land \neg\exists q''\in Q\ (q'<_q q'' \land q''<_q q)\}$ the set of its immediate predecessors.

We use homomorphisms\footnote{A homomorphism between two quasi-orders $P$ and $Q$ is a function $\varphi:P\to Q$ such that for any $p_0,p_1\in P$, if $p_0\leq_p p_1$, then $\varphi(p_0)\leq_q\varphi(p_1)$.} in order to compare structures. If there exists an injective homomorphism $\varphi:P\to Q$, then we write $P\injectivehomo Q;$ if it is injective and preserves immediate predecessors\footnote{A function $\varphi:P\to Q$ preserves immediate predecessors if, for any $p_0,p_1\in P$, whenever $p_0\in\Pred(p_1)$, then $\varphi(p_0)\in\Pred\big(\varphi(p_1)\big)$.}, then we write $P\rightarrowtail Q$. Notice that $P\rightarrowtail Q$ is more rigid than $P\injectivehomo Q$, hence, it describes more local behaviors.

If $q$ and $q'$ are elements of a quasi-order $Q$ such that $q\leq_q q' \text{ and }q'\leq_q q$, then we write $q\equiv_q q'$. The relation $\equiv_q$ is an equivalence relation whose equivalence classes are denoted by $[q]=\{q'\in Q\mid q\equiv_q q'\}$. The quotient set $\sfrac{Q}{\equiv_q}=\{[q]\mid q\in Q\}$ inherits the quasi-order $\leq_q.$ More precisely, we set $[q]\leq_q [q']$ if and only if $q\leq q'.$ The set $\sfrac{Q}{\equiv_q}$ equipped with $\leq_q$ is a poset, i.e., a quasi-order whose order-relation is a partial order\footnote{A quasi-order $(P,\leq_p)$ is a partial order if $\leq_p$ is antisymmetric, i.e., for any $p_0,p_1\in P$, $p_0\leq_p p_1$ and $p_1\leq_p p_0$ implies $p_0=p_1$.}.

We denote the class of countable posets by $\mathbb{P}$. If $P\in\mathbb{P},$ then we can always consider $\leq_p\ \subseteq \alpha\times\alpha$ where $\alpha \in \omega\cup\{\w\}$ via any bijection: $P\leftrightarrow \alpha$; so that all the posets we consider are posets on $P\in \omega \cup \{\w\}$. An order-embedding is a homomorphism between two posets $\varphi:P\to Q$ such that for any $p_0,p_1\in P$, $p_0\leq_p p_1$ if and only if $\varphi(p_0)\leq_q\varphi(p_1)$. Thus, order-embeddings are injective. The main posets studied in this article will be the set of \emph{finite} subsets of the integers ordered by inclusion $\big(\finite,\subseteq\!\!\big)$, and the set of \emph{infinite} subsets of the integers ordered by inclusion $\big(\infinite,\subseteq\!\!\big)$.

A 2-colored poset is a triple $\p=(P,\leq_p,\col_p)$ where $\leq_p$ is a partial order on $P$ and $\col_p:P\to 2$ is a 2-coloring. We usually use the letters $\p,\q$ for 2-colored posets. We also compare them via homomorphisms\footnote{A homomorphism between $\p,\q$ two 2-colored posets is a quasi-order homomorphism $\varphi: P \to Q$ such that for all $p\in P,$ $\col_p(p)=\col_q\big(\varphi(p)\big)$.}. If there exists a homomorphism from $\p$ to $\q$, then we write $\p\preccurlyeq_c \q;$ if this homomorphism is injective, then we write $\p\injectivehomo_c \q;$ if it is injective and preserves immediate predecessors, then we write $\p\rightarrowtail_c \q$. Notice that $\preccurlyeq_c$ is a quasi-order on 2-colored posets. We will denote by $\equiv_c$ the induced equivalence relation.
 
\subsection{Topological notations}

This article focuses on the study of a particular topological space first introduced by Scott as a universal model of the semantic of $\lambda$-calculus \cite{Scott1976}.

\begin{definition}
The \emph{Scott domain} is the power set of the integers $\mathcal{P}(\omega)$ equipped with the topology generated by the basis $$\big\{\OO_F \mid F\in \mathcal{P}_{<\omega}(\omega)\big\} \text{, where } \OO_F=\{x\subseteq \omega \mid F\subseteq x\}.$$
\end{definition}

The Scott domain is a non-metrizable -- in fact non-Hausdorff ($T_2$), and even non-Fréchet ($T_1$) -- compact space which is connected and Kolmogorov ($T_0$).

From now on and throughout this article, we use the notation $\Pw$ for the Scott domain; $F,G,H$ for finite subsets of $\omega;$ $x,y,z$ for arbitrary subsets of $\omega$; and $\A,\B,\C$ for subsets of $\Pw.$ 

Our ultimate goal is to study the topological complexity of subsets of $\scott.$ In metrizable spaces, this study begins with the definition of the Borel hierarchy (Section 11.B in \cite{Kechris1995}). However, the same construction would not work with $\scott$ for it is not metrizable. To overcome this obstacle, Selivanov introduced a new version of the Borel hierarchy for arbitrary spaces \cite{Selivanov2005bis, Selivanov2006}. This generalization extends the original one and induces a well-behaved hierarchy (see \cite{Debrecht2012} for more details). In the rest of this section, $\mathcal{T}$ denotes a topology on a set $X$. As usual, we denote by $X$ both the topological space and the underlying set without any risk of confusion.

\begin{definition} \label{defborel}
We define $\bsigma^0_1(X)=\mathcal{T},$ and for $1<\alpha<\w_1,$ $$\bsigma^0_\alpha(X)=\left\{\bigcup_{n\in\omega}(B_n\setminus B'_n)\ \middle\vert\  B_n,B'_n\in \bsigma^0_{\beta_n}(X),\ \beta_n<\alpha\right\}.$$ \\ We also define $\bpi^0_\alpha(X)=\{A\subseteq X \mid X\setminus A\in\bsigma^0_\alpha(X)\}, \bdelta^0_\alpha(X)=\bsigma^0_\alpha(X)\cap \bpi^0_\alpha(X)$ for $\alpha<\w_1.$ Finally, we define the {Borel sets} as $\bb(X)=\bigcup_{\alpha\in\omega_1} \bsigma^0_\alpha(X)$. \\
The Borel hierarchy on $X$ is the quasi-order $$\Big(\big\{\bsigma^0_\alpha(X),\bpi^0_\alpha(X)\big\}_{\alpha\in \w_1},\subseteq\!\! \Big).$$
\end{definition}

As customary in descriptive set theory, we consider the Hausdorff-Kuratowski difference hierarchy as a first refinement of the Borel hierarchy (see Section 22.E in \cite{Kechris1995}). Its definition relies on the difference operation.

\begin{definition}
If $0<\alpha<\omega_1$ and $(A_\beta)_{\beta<\alpha}$ is a sequence of subsets of $X$, then $$D_\alpha\big((A_\beta)_{\beta<\alpha}\big)=\bigcup\left\{A_\beta\setminus \cup_{\gamma<\beta} A_\gamma \ \middle\vert\,
\begin{tabular}{l}
\text{$\beta<\alpha,$ and} \\ \text{$\alpha$ and $\beta$ have different parities}
\end{tabular}\right\}\subseteq X.$$
%\text{$\alpha$ and $\beta$ have different parities and }\beta<\alpha\right\}\subseteq X$$ %where $\beta\nsim\alpha$ stands for $\beta$ and $\alpha$ have different parities. \\ 
If $0<\alpha,\beta<\omega_1$, then $$D_{\alpha}\big(\bsigma^0_\beta\big)(X)=\Big\{D_\alpha\big((A_\gamma)_{\gamma<\alpha}\big)\ \vert\  (A_\gamma)_{\gamma<\alpha}\subseteq \bsigma^0_\beta(X)\Big\}\subseteq \mathcal{P}(X).$$ Finally, we set $\check{D}_{\alpha}\big(\bsigma^0_\beta\big)(X)=\Big\{A\subseteq X \mid X\setminus A \in D_\alpha\big(\bsigma^0_\beta\big)(X)\Big\}.$ \\
The Hausdorff-Kuratowski difference hierarchy on $X$ is the quasi-order $$\Big(\big\{{D}_{\alpha}\big(\bsigma^0_\beta\big)(X),\check{D}_{\alpha}\big(\bsigma^0_\beta\big)(X)\big\}_{\alpha,\beta\in \w_1},\subseteq\!\! \Big).$$
\end{definition}

All Borel and Hausdorff-Kuratowski classes previously defined are closed under continuous preimages\footnote{A class of subsets $\Gamma(X)\subseteq \mathcal{P}(X)$ is closed under continuous preimages if for any $A\in \Gamma(X), f:X\to X$ continuous, then $f^{-1}[A]\in \Gamma(X)$.}. This suggests a natural further investigation of topological complexity through the lens of Wadge reducibility, a notion of comparison first studied thoroughly by Wadge in his PhD thesis \cite{Wadge1984}.

\begin{definition}
Let $A,B\subseteq X.$ The set $A$ is \emph{Wadge reducible} to $B$, written $A\leq_w B$, if there exists a continuous function $f:X\to X$ such that for all $x\in X$, $$x\in A\Longleftrightarrow f(x)\in B,$$ i.e., $f^{-1}[B]=A.$ \\
$A$ is \emph{Wadge equivalent} to $B$, written $A\equiv_w B,$ if $A\leq_w B$ and $B\leq_w A$ hold.
\end{definition}

Since both the identity and the composition of continuous functions are continuous, $\leq_w$ induces a quasi-order on the subset of $X$, and thus the binary relations $<_w,$ $\nleq_w$ and $\perp_w$ are well-defined. 

\begin{definition}
Let $X$ be any topological space and $\Gamma(X)\subseteq \mathcal{P}\big(X\big)$ be any class closed under continuous preimages. The {Wadge order on the $\Gamma$-subsets of $X$} is the quasi-order $\big(\Gamma(X),\leq_w\!\!\big).$
\end{definition}

For the equivalence relation $\equiv_w,$ we have a special terminology:

\begin{definition}
Let $X$ be any topological space, $A\subseteq X$ and $\Gamma(X)\subseteq \mathcal{P}\big(X\big)$ be any class closed under continuous preimages. \\ The Wadge degree of $A$ is its $\equiv_w$-equivalence class $[A]=\{B\subseteq A\mid A\equiv_w B\}$. \\ The {Wadge order on the $\Gamma$-degrees of $X$} is the poset $\sfrac{\big(\Gamma(X),\leq_w\!\!\big)}{\equiv_w}.$ 
\end{definition}

%Since we will essentially focus on subsets of $\scott$, we may simply write $\Gamma$ instead of $\Gamma(\Pw)$ whenever $\Gamma$ is a class of subsets of $\Pw.$

\subsection{Selivanov's toolbox}

We will restrict ourselves to the study of the quasi-order $\big(\bdelta^0_2(\scott),\leq_w\!\!\big).$ As mentioned in the Introduction, some results have already been obtained on this quasi-order in \cite{Selivanov2005bis} and \cite{Becher2015bis}. The main result of this article (Theorem \ref{injectivehomomorphism}) comes as a generalization of a construction introduced by Selivanov in \cite{Selivanov2005bis} that we recall here.

\begin{definition}[pp.56 in \cite{Selivanov2005bis}]
Let $T_\alpha$ be any well-founded tree of rank $\w\leq\alpha<\omega_1,$ $\xi:\omega^{<\omega}\to \omega$ be any injective mapping such that $\xi(\varnothing)=0,$ and $e:T_\alpha\to \finite$ be defined as $e(s)=\{\xi(t) \mid t\sqsubseteq s\}.$ The sets $Y_\alpha$ and $Z_\alpha$ are defined by:
\begin{enumerate}[nolistsep]
\item $Y_\alpha=e\big[T^1_\alpha\big]$, where $T^1_\alpha=\{s\in T_\alpha\mid \lh(s) \text{ is odd}\},$
\item $Z_\alpha=B(T_\alpha)\cup Y_\alpha,$ where $B(T_\alpha)=\{x\subseteq \omega \mid \forall s\in T_\alpha\ x\nsubseteq e(s)\}.$
\end{enumerate}
\end{definition}

In \cite{Selivanov2005bis}, it is shown that, given any $\w\leq\alpha<\omega_1,$  $Y_\alpha$ and $Z_\alpha$ are differences of $\alpha$ open sets, Wadge incomparable, and $\leq_w$-minimal among true differences of $\alpha$ open sets. More precisely:

\begin{theorem}[Propositions 5.9 and 6.4 in \cite{Selivanov2005bis}] \label{thmselivanov}
For $n\in\omega$, $\omega\leq\alpha,\beta<\omega_1$ and $\A\in \bdelta^0_2(\scott)\setminus \check{D}_\alpha(\bsigma^0_1)(\scott)$, we have:
\begin{enumerate}[nolistsep]
\item $D_n(\bsigma^0_1)(\scott)\setminus \check{D}_n(\bsigma^0_1)(\scott)$ and $\check{D}_n(\bsigma^0_1)(\scott)\setminus {D}_n(\bsigma^0_1)(\scott)$ form two incomparable Wadge degrees,
\item $Y_\alpha, Z_\alpha\in D_\alpha(\bsigma^0_1)(\scott)\setminus \check{D}_\alpha(\bsigma^0_1)(\scott),$
\item $Y_\alpha\perp_w Z_\beta,$
\item if $\omega\in \A,$ then $Z_\alpha\leq_w \A,$ 
\item if $\omega\notin\A,$ then $Y_\alpha\leq_w \A$.
\end{enumerate}
\end{theorem} 

The proof of Theorem \ref{thmselivanov} makes use of Selivanov's characterizations of the $\bdelta^0_2$-subsets and of the $D_\alpha\big(\bsigma^0_1\big)$-subsets of $\scott$. Since our proof will also require these characterizations, we first recall them. For this purpose, if $x, y\in \scott$ such that $x\subseteq y$, we introduce the notation $$[x,y]=\{z\in\Pw \mid x\subseteq z\subseteq y\}.$$

\begin{definition}[Definition 2.4 in \cite{Selivanov2005bis}]
$\A\subseteq \Pw$ is \emph{approximable} if, for all $x\in\A,$ there exists $F\in\mathcal{P}_{<\omega}(\omega)$ such that $F\subseteq x$ and $[F,x]\subseteq \A.$
\end{definition}

A subset $\A$ of $\scott$ is $\bdelta^0_2$ if the membership of any subset $x\subseteq \w$ to $\A$ can be approximated by a finite subset of $x$. More precisely:

\begin{theorem}[Theorem 3.12 in \cite{Selivanov2005bis}] \label{chardelta02}
Let $\A\subseteq \Pw.$ 
\begin{center}
$\A\in\bdelta^0_2(\scott)$ $\Longleftrightarrow$ $\A$ and $\scott \setminus\A$ are approximable.
\end{center}
\end{theorem}

The characterization of $D_\alpha\big(\bsigma^0_1\big)$-subsets of $\scott$ is a stratification of the previous result using the notion of a 1-alternating tree.

\begin{definition}[Definition 3.5 in \cite{Selivanov2005bis}]
Let $\A\subseteq \Pw$ and $0<\alpha<\omega_1$. 
A \emph{1-alternating tree for $\A$ of rank $\alpha$} is a homomorphism of quasi-orders $${f:(T,\sqsubseteq)\rightarrow (\mathcal{P}_{<\omega}(\omega),\subseteq)}$$ from a well-founded tree $T\subseteq \omega^{<\omega}$ of rank $\alpha$ such that: 
\begin{enumerate}[nolistsep]
\item $f(\varnothing)\in \A,$ and
\item for all $s^\frown \langle n\rangle\in T$, we have $\big(f(s)\in\A\leftrightarrow f(s^\frown \langle n\rangle)\notin \A\big).$
\end{enumerate}
\end{definition}

\begin{corollary} [Corollary 3.11 in \cite{Selivanov2005bis}]\label{chardiff}
Let $\A\subseteq \Pw$ and $0<\alpha<\omega_1$.
\[
\A\in D_\alpha(\bsigma^0_1)(\scott) \Longleftrightarrow 
\begin{cases}
\begin{tabular}{c}
\text{$\A\in\bdelta^0_2(\scott)$ and}  \\ \text{there is no 1-alternating tree for $A$ of rank $\alpha$.}
\end{tabular}
\end{cases}
\]
\end{corollary}

\section{The class $\boldsymbol{\mathbb{P}_{\lay}}$}

We define a class -- called $\mathbb{P}_{\lay}$ -- of countable 2-colored posets (Definition \ref{defrusp}) that will be mapped into the Wadge order on the subsets of the Scott domain in the next section. The definition of $\mathbb{P}_{\lay}$ will first be independent of $\scott$. Afterwards, we will give an order theoretic characterization of the elements of $\mathbb{P}_{\lay}$ that link them to $\scott$ (Proposition \ref{embedinscott}).

We begin with the naming of several posets that are useful for the definition of a subclass of $\mathbb{P}$ denoted by $\mathbb{P}_{\shrub}$. % -- the reason of this notation is that any finite subset $F$ of a poset $P\in \mathbb{P}_{\shrub}$ that admits an upper bound has a unique supremum (Proposition \ref{uniquesupremum}). 
In the following picture, we represent each poset $(P,\leq_p)$ with a directed graph $G=(P,\to)$. More precisely, if $p,q\in P,$ then $p\leq_p q$ if and only if there exists a finite sequence $(p_k)_{k\leq l}$ such that $p_0=p, p_l=q$ and for all $k<l,$ we have $p_k\to p_{k+1}.$

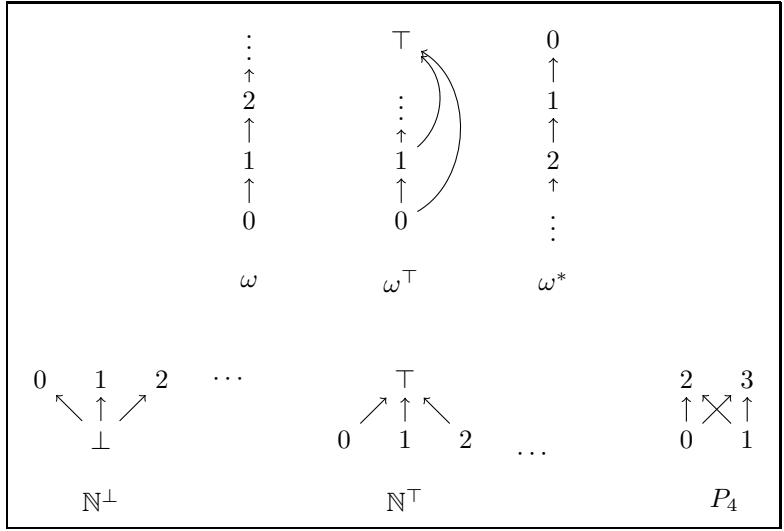
\begin{figure}[H]
\centering
\fbox{
\begin{tikzpicture}[scale=1,node distance=0.8cm and 0.8cm]
\node (a) at (0,0) {$\vdots$};
\node (b) [below of = a] {$2$};
\node (c) [below of = b] {$1$};
\node (d) [below of = c] {$0$};
\node (e) [below of = d] {$\w$};
\draw [<-] (a) -- (b);
\draw [<-] (b) -- (c);
\draw [<-] (c) -- (d);
\node (a4) at (2,0) {$\top$};
\node (b4) [below of = a4] {$\vdots$};
\node (c4) [below of = b4] {$1$};
\node (d4) [below of = c4] {$0$};
\node (e4) [below of = d4] {$\w^\top$};
\draw [<-] (a4) to [out=-40,in=40] (c4);
\draw [<-] (a4) to [out=-30,in=30] (d4);
\draw [<-] (b4) -- (c4);
\draw [<-] (c4) -- (d4);
\node (a1) at (4,0) {$0$};
\node (b1) [below of = a1] {$1$};
\node (c1) [below of = b1] {$2$};
\node (d1) [below of = c1] {$\vdots$};
\node (e1) [below of = d1] {$\w^*$};
\draw [<-] (a1) -- (b1);
\draw [<-] (b1) -- (c1);
\draw [<-] (c1) -- (d1);
\begin{scope}[yshift=-2.5cm, xshift=-8.25cm]
\node (a3) at (5.5,-2) {$0$};
\node (b3) [right of = a3] {$1$};
\node (c3) [right of = b3] {$2$};
\node at (8,-2) {$\cdots$};
\node (e3) [below of = b3] {$\bot$};
\node (f3) [below of = e3] {$\N^\bot$};
\draw [<-] (a3) -- (e3);
\draw [<-] (b3) -- (e3);
\draw [<-] (c3) -- (e3);
\node (a5) at (9.5,-2.8) {$0$};
\node (b5) [right of = a5] {$1$};
\node (c5) [right of = b5] {$2$};
\node at (12,-3) {$\cdots$};
\node (e5) [above of = b5] {$\top$};
\node (f5) [below of = b5] {$\N^\top$};
\draw [->] (a5) -- (e5);
\draw [->] (b5) -- (e5);
\draw [->] (c5) -- (e5);
\node (a2) at (14,-2) {$2$};
\node (b2) [right of = a2] {$3$};
\node (c2) [below of = a2] {$0$};
\node (d2) [right of = c2] {$1$};
\draw [<-] (a2) -- (c2);
\draw [<-] (b2) -- (c2);
\draw [<-] (a2) -- (d2);
\draw [<-] (b2) -- (d2);
\node at (14.5,-3.6) {$P_4$};
\end{scope}
\end{tikzpicture}
}
\caption{Samples of useful countable posets.} \label{Figure1}
\end{figure}

In \cite{Selivanov2005bis}, Selivanov worked with well-founded trees in order to construct subsets of $\scott.$ We will generalize this construction to a larger class of posets that we call \emph{shrubs} and that share some of the properties of well-founded trees. For this purpose, we make use of the classical notion of \emph{bounded completeness} that occurs in domain theory.

\begin{definition}
A subset $S\subseteq P$ of a poset is \emph{bounded} if there exists an element $p'\in P$ -- called an \emph{upper bound} -- such that, for any $p\in S,$ we have $p\leq_p p'.$ If the set of all upper bounds of $S\subseteq P$ has a unique $\leq_p$-minimal element -- i.e., if there exists an upper bound $s_S\in P$ of $S$ such that, for any other upper bound $p'\in P$ of $S,$ we have $s_S\leq_p p'$ -- then $s_S$ is called \emph{the supremum} of $S$ in $P.$ \\
A poset $P$ is \emph{bounded complete} if any bounded $S\subseteq P$ admits a -- necessarily unique -- supremum.
\end{definition}

Notice that $P_4$ is a typical example of a poset which is not bounded complete, while all the other examples of Figure \ref{Figure1}, as well as  $(\finite,\subseteq)$ and $(\scott,\subseteq)$ are examples of bounded complete posets.

\begin{definition} \label{defrusp}
The class of all {shrubs} $\mathbb{P}_{\shrub}\subseteq \mathbb{P}$ is the class of all countable posets $P\in \mathbb{P}$ that satisfy:
\begin{enumerate}[noitemsep]
\item $\w\notinjectivehomo P,$
\item for all $p\in P,\ \Card(\Pred(p))<\aleph_0,$
\item there exists a $\leq_p$-minimal element $\bot,$
\item $P$ is bounded complete.  
\end{enumerate}
\end{definition}

Well-founded trees and $\N^{\bot}$ are typical examples of shrubs. More involved ones will be constructed in the proof of Theorem \ref{illfounded} (Figure \ref{posetpn}) and of Theorem \ref{antichain} (Figure \ref{posetqn}). To the contrary, $\omega,$ $\omega^\top\!\!,$ $\w^*\!\!,$ $\N^\top\!\!,$ and $P_4$ are typical examples of posets that are not shrubs. 

In the next proposition, we give alternative characterizations to the second item of the previous definition. In particular, we show that the posets we just defined can be embedded into $\finite$. We also give an alternative characterization of this second item that exclusively depends on morphisms between posets.

\begin{proposition} \label{embedinscott}
If $P\in\mathbb{P},$ then the following are equivalent: 
\begin{enumerate}[noitemsep]
\item $\text{for all $p\in P,\ \Card(\Pred(p))<\aleph_0$}$,
\item $P\injectivehomo \finite$,
\item $(\omega^\top \notinjectivehomo P),\ (\omega^* \notinjectivehomo P)$ and $(\N^\top \notinjectivehomo P)$.
\end{enumerate}
\end{proposition}

\begin{proof} ~
\begin{description}
\item[(1. $\boldsymbol{\Rightarrow}$ 2.):] 
We consider $P\in \omega\cup\{\omega\}$ and define a function:
\begin{align*}
e: P&\to \finite \\
k&\mapsto \{n\mid n\leq_p k\}.
\end{align*}
If $k\leq_p l,$ then by transitivity of $\leq_p,$ we get $e(k)\subseteq e(l)$. If $k\neq l$, we consider the two cases $k<_pl$ and $k\perp_p l$ (the third case $l<_pk$ is the same as the case $k<_pl$). In both cases, $l\in e(l)\setminus e(k).$ Therefore, we obtain that $e$ is an injective homomorphism that witnesses $P\injectivehomo \finite$.
\item[(2. $\boldsymbol{\Rightarrow}$ 3.):] If $\varphi: Q\injectivehomo P$, then for all $q\in Q,$ the injectivity of $\varphi$ implies $\Card\big(\Pred(q)\big)\leq\Card\big(\Pred(\varphi(q))\big).$
Towards a contradiction, we assume that $(\omega^\top \injectivehomo P) \lor (\omega^* \injectivehomo P)\lor (\N^\top\injectivehomo P)$ holds. We get a contradiction for each one of these situations: 
\begin{enumerate}[nolistsep]
\item if $\omega^\top \injectivehomo P$, then $\Card\big(\Pred(\varphi(\top))\big)= \aleph_0,$
\item if $\omega^* \injectivehomo P$, then $\Card\big(\Pred(\varphi(0))\big)= \aleph_0,$
\item if $\N^\top \injectivehomo P$, then $\Card\big(\Pred(\varphi(\top))\big)= \aleph_0.$
\end{enumerate}
However, there exists no $F\in \finite$ such that $\Card\big(\Pred(F)\big)= \aleph_0.$
\item[(3. $\boldsymbol{\Rightarrow}$ 1.):]
Towards a contradiction, we pick $p\in P$ such that $\Card(\Pred(p))=\aleph_0.$ We consider three different cases.
\begin{enumerate}[label=(\alph*),nolistsep]
\item Suppose there exists $q_0<_p p$ such that there exists no immediate predecessor $p'$ of $p$ satisfying $q_0\leq_p p'.$ Hence, there exists $q_1<_p p$ such that $q_0<_p q_1$. We continue the process to construct a sequence $(q_n)_{n\in\omega}$ witnessing $\omega^\top \injectivehomo P$ via the mapping: $\top\mapsto p,$ and $n\mapsto q_n$ for any $n\in\w$.
\item Suppose there exist infinitely many immediate predecessors $(q_n)_{n\in\omega}$ of $p\in P,$ then the mapping: $\top\mapsto p,$ and $n\mapsto q_n$ for any $n\in\w,$ witnesses $\N^\top\injectivehomo P$.
\item Suppose that we are not in the situations (a) and (b); then, by the pigeonhole principle, there exists $q_0$ an immediate predecessor of $p$ such that $\Card(\Pred(q_0))=\aleph_0$. If we replace $p$ with $q_0$ and start the proof again, either we get a contradiction from (a) or (b), or we exhibit $q_1$ an immediate predecessor of $q_0$ such that $\Card(\Pred(q_1))=\aleph_0.$ By an infinite iteration of this process, we obtain a sequence $(q_n)_{n\in\omega}$ witnessing $\omega^* \injectivehomo P$ via the mapping: $0\mapsto p,$ and $n\mapsto q_{n-1}$ for any $n\in\w^+.$
\end{enumerate}
\end{description}
\end{proof}

In the next section, we will associate a subset $\A_\p$ of $\scott$ to some countable 2-colored posets $\p,$ where the color $1$ will correspond to elements inside $\A_\p.$

In Figure \ref{figure2}, we give a name to some specific 2-colored posets: the nodes of the form $\bullet$ and $\circ$ correspond to color 1 and color 0, respectively.

\begin{figure}[H]
\centering
\fbox{
\begin{tikzpicture}
\node (a) at (0.2,-0.2) {$\circ$};
\node (b) at (0.8,-0.2) {$\circ$};
\node (c) at (0.5,-0.8) {$\bullet$};
\node at (-0.5,-0.5) {$\boldsymbol{\vee_1^0}:$};
\draw [->] (c) -- (a);
\draw [->] (c) -- (b);
\node (a1) at (3.2,-0.8) {$\circ$};
\node (b1) at (3.8,-0.8) {$\circ$};
\node (c1) at (3.5,-0.2) {$\bullet$};
\node (d1) at (2.5,-0.5) {$\boldsymbol{\wedge^1_0}:$};
\draw [<-] (c1) -- (a1);
\draw [<-] (c1) -- (b1);
\node (a2) at (6,-0.2) {$\bullet$};
\node (b2) at (6,-0.8) {$\bullet$};
\draw [<-] (a2) -- (b2);
\node (d2) at (5.5,-0.5) {$\boldsymbol{\mid_1^1}:$};
\end{tikzpicture}}
\caption{Samples of useful 2-colored countable posets.} \label{figure2}
\end{figure}
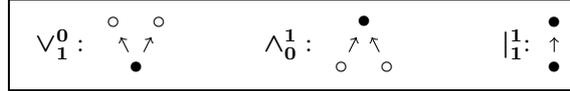

The next definition introduces the class of \emph{embeddable posets}. We will later associate a subset of $\scott$ to each such 2-colored poset.

\begin{definition} \label{defplay}
The class of \emph{embeddable posets} $\mathbb{P}_{\lay}$ is the class of countable 2-colored posets $\p=(P,\leq_p,\col_p)$ such that $(P,\leq_p)\in \mathbb{P}_{\shrub}$ and whose coloring satisfies:
\begin{enumerate}[noitemsep]
\item $\col_p(\bot)=0$,
\item for all $k\in P$ $\leq_p$-maximal, $\col_p(k)=1,$
\item $(\boldsymbol{\vee_1^0} \not\rightarrowtail_c \p),\ (\boldsymbol{\wedge^1_0} \not\rightarrowtail_c \p)$ and $(\ \!\boldsymbol{\mid_1^1} \not\rightarrowtail_c \p)$.
\end{enumerate}
\end{definition}

If $\p$ is an embeddable poset, then the nodes of color 1 are isolated. Indeed, if $\p\in \mathbb{P}_{\lay}$, $p\in P$ and $\col_p(p)=1,$ then $p$ has a unique immediate predecessor; and $p$ has at most one immediate successor\footnote{If $\p$ is an embeddable poset, $p\in P$ is an immediate successor of $p'\in P$ if $p'\in\Pred_{\imm}(p).$}, depending on whether $p$ is $\leq_p$-maximal or not. Moreover,  if they exist, they both have color 0. Thus, we introduce the following notations.

\begin{notation}\label{p-p+}
For $\p\in \mathbb{P}_{\lay}$, $p\in P$ and $\col_p(p)=1,$ we denote by $p^-$ its unique immediate predecessor; and, if it exists, by $p^+$ its unique immediate successor. We have $\col_p(p^-)=\col_p(p^+)=0$.
\end{notation}

This means that the direct neighborhood -- composed of all immediate predecessors and all immediate successors -- of every node of color 1 is of one of the following form, depending on whether it is $\leq_p$-maximal or not:

\begin{figure}[H]
\centering
\fbox{
\begin{tikzpicture}
\node (a1) at (3,-0.7) {$\bullet$};
\node at (2.5,-0.7) {$p$};
\node (b1) at (3,-1.3) {$\circ$};
\node at (2.5,-1.3) {$p^-$};
\draw [<-] (a1) -- (b1);
\node (a2) at (6,-0.4) {$\circ$};
\node (b2) at (6,-1) {$\bullet$};
\node (c2) at (6,-1.6) {$\circ$};
\node at (5.5,-0.4) {$p^+$};
\node at (5.5,-1) {$p$};
\node at (5.5,-1.6) {$p^-$};
\draw [<-] (a2) -- (b2);
\draw [<-] (b2) -- (c2);
\end{tikzpicture}
}
\caption{The two possible direct neighborhoods of any $p\in P,$ where $\p\in \mathbb{P}_{\lay}$ and $\col_p(p)=1.$ The first case occurs when $p$ is $\leq_p$-maximal, and the second one when $p$ is not.}
\end{figure}
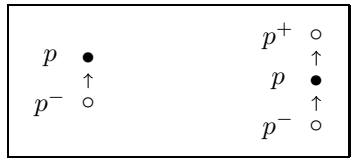

\section{An order-embedding into the Wadge order}

In this section, we associate a subset $\A_\p\in \bdelta^0_2(\scott)$ to each countable 2-colored poset $\p$ which is embeddable, and show that this association is such that, for any $\p,\q\in \mathbb{P}_{\lay},$ $\p\preccurlyeq_c \q$ if and only if $\A_\p\leq_w \A_\q$ (Lemma \ref{lemmainjectivehomomorphism}). As a consequence, we get our main result that there exists an order-embedding $\sfrac{({\mathbb{P}_{\reg}},\preccurlyeq_c)}{\equiv_c} \to \sfrac{(\bdelta^0_2(\Pw),\leq_w)}{\equiv_w}$ (Theorem \ref{injectivehomomorphism}).

We first need to label the elements of any embeddable poset.

\begin{definition} \label{labeling}
Let $\p\in \mathbb{P}_{\reg}$ so that $P\in \w\cup\{\w\}$ has a $\leq_p$-minimal element $m=\bot$ for some $m\in \w$. The labeling $l_p$ on $P$ is defined by:
\begin{align*}
l_p: P &\to \PP_{<\omega}(\omega) \\
\bot&\mapsto \emptyset,\\
n&\mapsto \bigcup_{\substack{ k\leq_p n}} \{k\}.
\end{align*}
\end{definition}
We notice that $l_p$ is injective. Therefore, for every $F\in\finite$ in the range of $l_p,$ $l_p^{-1}(F)$ is well-defined.

We then associate a subset of the Scott domain to any embeddable poset through the labeling given by Definition \ref{labeling}.

\begin{definition} \label{association}
Let $\p\in \mathbb{P}_{\reg},$ we define the subset $\A_\p\subseteq \scott$ as:
\begin{align*}
\A_\p&= l_p\big[\col^{-1}_p[\{1\}]\big]
\\ &= \big\{x\subseteq \omega \mid \exists p\in P\ (\col_p(p)=1\land l_p(p)=x)\big\}.
\end{align*}
We also denote by $\C(\A_\p)$ the set of all finite sets of integers contained in the labeling of an element of $P$:
$$\C(\A_\p)=\big\{F\subseteq \omega \mid \exists p\in P\ F\subseteq l_p(p)\big\}.$$ 
\end{definition}

The next lemma gathers two crucial observations that arise from the construction given by Definition \ref{association}.

\begin{lemma}\label{observations}
Let $\p\in \mathbb{P}_{\lay}$ and $F\in \finite$.
\begin{enumerate}[nolistsep]
\item If $F\in \C(\A_\p),$ then $\{p\in P \mid l_p(p)\subseteq F\}$ has an upper bound in $\p$. \\
By Definition $\ref{defrusp}$, it has a unique supremum denoted by $s_F\in P$.
\item $F\in \A_\p \Leftrightarrow \big(\col_p(s_F)=1\land l_p(s_F)=F\big).$
\end{enumerate}
\end{lemma}

\begin{proof}
\begin{enumerate}
\item Since $F\in \C(\A_{\p})=\big\{F\subseteq \omega \mid \exists p\in P\ F\subseteq l_p(p)\big\},$ there exists $p_0\in P$ such that $F\subseteq l_p(p_0).$ Thus, $p_0\in P$ is an upper bound of $\{p\in P \mid l_p(p)\subseteq F\}$.
\item Assume first that $F\in \A_\p\subseteq \C(\A_\p).$ Then, there exists $p_0\in P$ such that $\col_p(p_0)=1$ and $l_p(p_0)=F.$ It implies that $p_0\leq_p s_F.$ Since $s_F$ is the supremum of $\{p\in P \mid l_p(p)\subseteq F\}$ and $s_F$ has a unique immediate predecessor, we have $s_F\in \{p\in P \mid l_p(p)\subseteq F\}.$ Thus $F=l_p(p_0)\subseteq l_p(s_F)\subseteq F.$ By injectivity of $l_p,$ we obtain $s_F=p_0$ and $\col_p(s_F)=1$. \\
Conversely, from the very definition of $\A_\p$, we have $\col_p(s_F)=1$ and $l_p(s_F)=F$, which implies that $F\in \A_\p$.
\end{enumerate}
\end{proof}

The rest of this section consists in proving that the correspondence $\p\mapsto \A_\p$ satisfies that $\A_\p\in\bdelta^0_2(\scott)$ and for any $\p,\q\in \mathbb{P}_{\lay},$ $\p\preccurlyeq_c \q$ if and only if $\A_\p\leq_w \A_\q$. For this, we need a result which claims that a continuous mapping from $\scott$ to itself is completely determined by its behavior on $\mathcal{P}_{<\omega}(\w)$.

\begin{lemma}[Exercice 5.1.62 in \cite{Goubaultlarrecq2013}] \label{goubault}
Given any $\subseteq$-increasing mapping $f:\mathcal{P}_{<\omega}(\omega)\to \Pw,$ there exists a \emph{unique} continuous extension of $f$ to the whole Scott domain. This extension is given by 
\begin{align*}
\hat{f}: \Pw&\to \Pw \\
x&\mapsto \bigcup_{n\in\omega} f\big(x\cap n\big).
\end{align*}
\end{lemma}

\begin{proof} ~
\begin{description}
\item[Existence:]
It suffices to prove that $\hat{f}$ is continuous. Observe that, for all $x\in\Pw,$ the sequence $\big(f(x\cap n)\big)_{n\in\omega}$ is $\subseteq$-increasing. Let $F\in\mathcal{P}_{<\omega}(\omega)$ such that $\OO_F$ is a basic open set. If $x\in \hat{f}^{-1}\big[\OO_F\big],$ then $F\subseteq \hat{f}(x).$ Since $F$ is finite, there exists $n_0\in\omega$ such that $F\subseteq f\big(x\cap n_0\big)$. We obtain $$x\in \OO_{x\cap n_0}\subseteq \hat{f}^{-1}\big[\OO_F\big],$$ which shows that $\hat{f}$ is continuous.

\item[Uniqueness:]
Observe that a continuous function has to be $\subseteq$-increasing on the whole domain. This follows from the $\subseteq$-upward closure of the open subsets of $\scott$. Let $g:\Pw\to\Pw$ be any continuous extension of $f.$ Given any $x\in \Pw,$ consider $k\in g(x).$ By continuity, there exists $F\in \mathcal{P}_{<\omega}(\omega)$ such that $x\in \OO_F\subseteq g^{-1}\big[\OO_{\{k\}}\big]$. Because $F$ is finite, there exists $n_0\in\omega$ such that $x\cap n_0\in \OO_F.$ Thus, $k\in g\big(x\cap n_0\big)=f\big(x\cap n_0\big)\subseteq \hat{f}(x).$ The exact same reasoning works if $g$ and $\hat{f}$ are swapped. Hence, we conclude that for all $x\in\Pw$, we have $\hat{f}(x)=g(x).$
\end{description}
\end{proof} 

We are now ready for our main proof.

\begin{lemma}\label{lemmainjectivehomomorphism}
The following mapping
\begin{align*}
H:({\mathbb{P}_{\reg}},\preccurlyeq_c) &\to (\bdelta^0_2(\Pw),\leq_w)\\ \p&\mapsto \A_\p
\end{align*}
satisfies that for any $\p,\q \in {\mathbb{P}_{\reg}}$, we have $$\p\preccurlyeq_c\q \text{ if and only if } \A_\p\leq_w\A_\q.$$
\end{lemma}

\begin{proof}

The proof is divided into the three Claims \ref{claim23}, \ref{claim24} and \ref{claim25}. The first two claims show that $H$ is well-defined and order-preserving, while the third one completes the proof.

\begin{claim} \label{claim23}
If $\p\in \mathbb{P}_{\reg},$ then $\A_\p\in \bdelta^0_2(\scott).$
\end{claim}

\begin{claimproof}
We show that $\A_\p$ is both approximable and co-approximable. $\A_\p$ is approximable because $\A_\p\subseteq \PP_{<\w}(\w)$. For co-approximability, we proceed by contradiction and suppose that $\A_\p$ is not co-approximable for some $x\in\scott \setminus \A_\p$ infinite. So, we fix $F_0\in [\emptyset,x]\cap \A_\p$ and set $p_0=l_p^{-1}(F_0).$ Assume $F_n$ and $p_n$ are already constructed. Since $\A_\p$ is not co-approximable, there exists $F_{n+1}\in \big([F_n,x]\setminus\{F_n\}\big)\cap \A_\p.$ We set $p_{n+1}=l_p^{-1}(F_{n+1}).$ It follows that the function
\begin{align*}
\varphi:\w &\to P \\ n&\mapsto p_n
\end{align*}
witnesses $\omega\injectivehomo P,$ a contradiction.
\end{claimproof}

\begin{claim} \label{claim24}
If $\p,\q\in \mathbb{P}_{\reg}$ and $\p\preccurlyeq_c \q$, then $\A_\p\leq_w \A_\q.$
\end{claim}

\begin{claimproof}
Suppose that $\p\preccurlyeq_c \q$ is witnessed by $\varphi: P\to Q.$ Consider the function:
\begin{align*}
f_\varphi : \PP_{<\omega}(\omega)&\to \Pw \\
F&\mapsto 
\begin{cases} 
l_q\big(\varphi(s_F)\big) & \text{ if $F\in \C(\A_\p)\land \col_p(s_F)=0,$} \\ 
l_q\big(\varphi(s_F)\big) & \text{ if $F\in \C(\A_\p)\land \col_p(s_F)=1 \land F=l_p(s_F),$} \\ 
l_q\big(\varphi(s_F^-)\big) & \text{ if $F\in \C(\A_\p)\land \col_p(s_F)=1 \land F\subsetneq l_p(s_F),$} \\ 
l_q\big(\varphi(s_F^+)\big) & \text{ if $F\in \C(\A_\p)\land \col_p(s_F)=1 \land F\nsubseteq l_p(s_F),$} \\ 
\omega &\text{ otherwise,}
\end{cases}
\end{align*}
where $s_F$ is defined as in Lemma \ref{observations}; $s_F^-$ and $s_F^-$ are defined as in Notation \ref{p-p+}; and $s^+_F$ is replaced by $\omega$ whenever $s_F$ is a maximal element in $(P,\leq_p).$

We show that the function $\hat{f}_\varphi$ given by Lemma \ref{goubault} satisfies $\hat{f}_\varphi^{-1}\big[\A_\q\big]=\A_\p.$ First, for $\hat{f}_\varphi$ to exist, we need $f_\varphi$ to be increasing. Let $F,G \in \PP_{<\omega}(\omega)$ be such that $F\subseteq G.$ We have several cases to check:
\begin{enumerate}[noitemsep]
\item if $G\notin \C(\A_\p),$ then $f_\varphi(F)\subseteq f_\varphi(G)=\omega.$
\end{enumerate}
Since $G\in \C(\A_\p)$ implies $F\in \C(\A_\p),$ we now suppose $F,G\in \C(\A_\p)$ and thus $s_F\leq_p s_G.$
\begin{enumerate}[noitemsep,resume]
\item if $\col_p(s_F)=\col_p(s_G)=0,$ then $f_\varphi(F)=l_q\big(\varphi(s_F)\big)\subseteq l_q\big(\varphi(s_G)\big)=f_\varphi(G),$

\item if $\col_p(s_F)=0$ and $\col_p(s_G)=1,$ then $f_\varphi(F)=l_q\big(\varphi(s_F)\big)\subseteq l_q\big(\varphi(s_G^-)\big)\subseteq f_\varphi(G),$

\item if $\col_p(s_F)=1$ and $\col_p(s_G)=0,$ then $f_\varphi(F)\subseteq l_q\big(\varphi(s_F^+)\big)\subseteq l_q\big(\varphi(s_G)\big)=f_\varphi(G),$

\item if $\col_p(s_F)=\col_p(s_G)=1$ and $s_F\neq s_G,$ then there exists $p\in P$ such that $s_F<_p p <_p s_G$ holds, because there exist no two consecutive nodes colored by $1$. Therefore $f_\varphi(F)\subseteq l_q\big(\varphi(s_F^+)\big)\subseteq l_q\big(\varphi(s_G^-)\big)=f_\varphi(G).$
\end{enumerate}
It only remains to consider the cases where $\col_p(s_F)=\col_p(s_G)=1,$ and $s_F= s_G$:
\begin{enumerate}[noitemsep,resume]
\item if $F,G\in \A_\p,$ then $f_\varphi(F)=l_q\big(\varphi(s_F)\big)=l_q\big(\varphi(s_G)\big)=f_\varphi(G),$

\item if $F\in \A_\p$ and $G\notin \A_\p,$ then $f_\varphi(F)=l_q\big(\varphi(s_F)\big)\subseteq l_q\big(\varphi(s_F^+)\big)=f_\varphi(G),$

\item if $F\notin \A_\p$ and $G\in \A_\p,$ then $f_\varphi(F)=l_q\big(\varphi(s_F^-)\big)\subseteq l_q\big(\varphi(s_F)\big)=f_\varphi(G),$

\item if $F,G\notin \A_\p$ and $F\subsetneq l_p(s_F),$ then $f_\varphi(F)=l_q\big(\varphi(s_F^-)\big)\subseteq f_\varphi(G),$

\item if $F,G\notin \A_\p$ and $F\nsubseteq l_p(s_F),$ then $f_\varphi(F)=l_q\big(\varphi(s_F^+)\big)=f_\varphi(G).$
\end{enumerate}
This finishes the proof that $f_\varphi:\PP_{<\omega}(\omega)\to \Pw$ is increasing. It follows from Lemma \ref{goubault}, that $f_\varphi$ has a continuous extension $\hat{f}_\varphi:\Pw\to \Pw$. We distinguish between three different cases to show that $\hat{f}_\varphi^{-1}\big[\A_\q\big]=\A_\p$.

\begin{description}
\item[$\boldsymbol{x\in \PP_{\omega}(\omega):}$] because $\A_\p\subseteq \PP_{<\omega}(\omega),$ we have $x\notin \A_\p.$ Suppose, towards a contradiction, that $\hat{f}_\varphi(x)\in \A_\q.$ Since $\A_\q\subseteq \PP_{<\omega}(\omega),$ there exist $F\in \PP_{<\omega}(\omega)$ and $n\in\omega,$ such that $\hat{f}_\varphi(x)=F\in \A_\q$ and $f_\varphi\big(x\cap m\big)=F$ both hold for all $m\geq n.$ We then notice that
\begin{align*}
f_\varphi(G)\in \A_\q&\Rightarrow G\in \C(\A_\p)\land \col_p(s_G)=1 \land G=l_p(s_G) \\
&\Rightarrow G\in \A_\p.
\end{align*}
Where the first implication comes from the definition of $f_\varphi$ and the second from Lemma \ref{observations}.
We obtain that $x\cap m\in \A_\p$ holds for all $m\geq n$, this implies $\col_p\big(l_p^{-1}(x\cap m)\big)=1$, and since $x$ is infinite, we can extract a subsequence of $\big(l_p^{-1}(x\cap m)\big)_{m\in\omega}$ witnessing $\omega\injectivehomo P$, a contradiction.

\item[$\boldsymbol{F\in \PP_{<\omega}(\omega) \setminus \C(\A_\p):}$] $F\notin \A_\p$ holds by the very definition of $\C(\A_\p)$. Hence, we have $\omega=f_\varphi(F)= \hat{f}_\varphi (F) \notin \A_\q.$

\item[$\boldsymbol{F\in \C(\A_\p):}$] Suppose first that $F\in \A_\p.$ By Lemma \ref{observations}, $\hat{f}_\varphi(F)=l_q\big(\varphi(s_F)\big)$ is satisfied. Moreover, from $\col_q\big(\varphi(s_F)\big)=1$, we get $\hat{f}_\varphi(F)\in\A_\q.$ \\ Suppose now that $F\notin \A_\p.$ By Lemma \ref{observations}, there are three cases:
\begin{enumerate}[nolistsep]
\item if $\col_p(s_F)=0,$ then $\col_q\big(\varphi(s_F)\big)=0$ which implies $\hat{f}_\varphi(F)\notin\A_\q,$
\item if $\col_p(s_F)=1$ and $F\subsetneq l_p(s_F),$ then $\col_q\big(\varphi(s_F^-)\big)=0$ which implies $\hat{f}_\varphi(F)\notin\A_\q,$
\item if $\col_p(s_F)=1$ and $F\nsubseteq l_p(s_F),$ then $\col_q\big(\varphi(s_F^+)\big)=0$ which implies $\hat{f}_\varphi(F)\notin\A_\q$.
\end{enumerate}
\end{description}
\end{claimproof}

\begin{claim} \label{claim25}
If $\p,\q\in \mathbb{P}_{\reg}$ and $\A_\p\leq_w \A_\q$, then $\p\preccurlyeq_c \q.$
\end{claim}

\begin{claimproof}
We assume that $\A_\p\leq_w \A_\q$ is witnessed by some continuous function $f:\Pw\to \Pw$. We describe a reduction which witnesses $\p\preccurlyeq_c \q$. First, we need a few observations. Let $p\in P.$ Since $\omega\notinjectivehomo P$ and all $\leq_p$-maximal elements have color 1, there exists $p'\in P$ such that both $p\leq_p p'$ and $\col_p(p')=1$ hold. Therefore, $f\big(l_p(p')\big)\in \A_\q.$ Hence, for all $p\in P$, we have $f\big(l_p(p)\big)\in \C(\A_\q).$ We also define, for all $p\in P$, the set $$Q_p=\big\{q\in Q \mid l_q(q)\subseteq f\big(l_p(p)\big)\big\}.$$ Since $f\big(l_p(p)\big)\in \C(\A_\q)$ holds, Lemma \ref{observations} yields the existence of a unique supremum $t_p$ of $Q_p$ in $Q.$ \\
We define a mapping:
\begin{align*}
\varphi : P &\to Q \\
p&\mapsto 
\begin{cases}  
t_p & \text{ if $f\big(l_p(p)\big)\in \A_\q,$} \\
t_p & \text{ if $f\big(l_p(p)\big)\notin \A_\q \land \col_q(t_p)=0,$} \\
t_p^- & \text{ if $f\big(l_p(p)\big)\notin \A_\q \land \col_q(t_p)=1 \land l_q(t_p)\subsetneq f\big(l_p(p)\big),$} \\
t_p^+ & \text{ if $f\big(l_p(p)\big)\notin \A_\q \land \col_q(t_p)=1 \land l_q(t_p)\nsubseteq f\big(l_p(p)\big),$} 
\end{cases}
\end{align*}
where $t_p^-$ and $t_p^+$ are defined as in Notation \ref{p-p+}.

For $\varphi$ to be well-defined, we need $t^+_p$ not to occur whenever $t_p$ is a $\leq_q$-maximal element. So, suppose $t_p$ is a $\leq_q$-maximal element. Since $\col_q(t_p)=1$, then $t_p\in Q_p$ for it has a unique immediate predecessor. Thus, $l_q(t_p)\subseteq f\big(l_p(p)\big)$ holds, which shows that $t^+_p$ does not occur in this case. \\
Since for every $p\in P$ we have
$$\col_p(p)=1 \Leftrightarrow l_p(p)\in \A_\p \Leftrightarrow f\big(l_p(p)\big)\in \A_\q,$$ it follows from the definition of $\varphi,$ that for all $p\in P$ we also have $\col_p(p)=\col_q(\varphi(p)).$ Therefore, it only remains to show that $\varphi$ is order-preserving. Suppose $p\leq_p p',$ we get $t_p\leq_q t_{p'}.$ We proceed with cases:
\begin{enumerate}[noitemsep]
\item if $\col_q(t_p)=\col_q(t_{p'})=0,$ then $\varphi(p)=t_p\leq_q t_{p'}=\varphi(p'),$

\item if $\col_q(t_p)=0$ and $\col_q(t_{p'})=1,$ then $\varphi(p)=t_p\leq_q t_{p'}^-\leq_q \varphi(p'),$

\item if $\col_q(t_p)=1$ and $\col_q(t_{p'})=0,$ then $\varphi(p)\leq_q t_p^+\leq_q t_{p'}=\varphi(p'),$

\item if $\col_q(t_p)=\col_q(t_{p'})=1$ and $t_p\neq t_{p'},$ then there exists some $q\in Q$ that satisfies $t_p<_q q <_q t_{p'}$. This finally leads to $\varphi(p)\leq_q t_p^+\leq_q t_{p'}^-=\varphi(p').$
\end{enumerate}
It only remains to consider the cases where $\col_q(t_p)=\col_q(t_{p'})=1,$ and $t_p= t_{p'}$:
\begin{enumerate}[noitemsep,resume]
\item if $\col_p(p)=\col_p(p')=1,$ then $\varphi(p)=t_p=t_{p'}=\varphi(p'),$

\item if $\col_p(p)=1$ and $\col_p(p')=0,$ then $\varphi(p)=t_p\leq_q t_p^+=\varphi(p'),$

\item if $\col_p(p)=0$ and $\col_p(p')=1,$ then $\varphi(p)=t_p^-\leq_q t_p=\varphi(p'),$

\item if $\col_p(p)=\col_p(p')=0$ and $l_q(t_p)\subsetneq f\big(l_p(p)\big),$ then $\varphi(p)=t_p^-\leq_q \varphi(p'),$

\item if $\col_p(p)=\col_p(p')=0$ and $l_q(t_p)\nsubseteq f\big(l_p(p)\big),$ then $\varphi(p)=t_p^+=\varphi(p').$
\end{enumerate}
This concludes the proof that $\varphi$ witnesses $\p\preccurlyeq_c \q$.
\end{claimproof}

So, Claim \ref{claim23} proves that the mapping $H:\p\mapsto \A_\p$ is a well-defined mapping from $({\mathbb{P}_{\reg}},\preccurlyeq_c)$ to $(\bdelta^0_2(\Pw),\leq_w)$, and we conclude from the Claims \ref{claim24} and \ref{claim25} that for any $\p,\q\in {\mathbb{P}_{\reg}}$, $\p\preccurlyeq_c\q$ if and only if $\A_\p\leq_w\A_\q$.
\end{proof}

The previous lemma almost immediately yields the main result:

\begin{theorem}\label{injectivehomomorphism}
The following mapping is an order-embedding:
\begin{align*}
\sfrac{({\mathbb{P}_{\reg}},\preccurlyeq_c)}{\equiv_c} &\to \sfrac{(\bdelta^0_2(\Pw),\leq_w)}{\equiv_w}\\ [\p]&\mapsto [\A_\p].
\end{align*}
\end{theorem}

\begin{proof}
By Lemma \ref{lemmainjectivehomomorphism} and the definition of the order on quotient sets, it is clear that for any $[\p],[\q]\in \sfrac{({\mathbb{P}_{\reg}},\preccurlyeq_c)}{\equiv_c}$, we have $[\p]\preccurlyeq_c[\q]$ if and only if $[\A_\p]\leq_w[\A_\q]$. Moreover, if $[\A_\p]=[\A_\q],$ then $\A_\p\equiv_w\A_\q$, and by Lemma \ref{lemmainjectivehomomorphism}, we have $\p\equiv_c \q$, hence $[\p]=[\q]$. Thus, the mapping $[\p]\mapsto [\A_\p]$ is an order-embedding.
\end{proof}

\section{A reduction game on $\boldsymbol{\mathbb{P}}$} \label{sectionduparc}

This section introduces a game characterization of reductions on 2-colored posets. This characterization and the order-embedding given in Theorem \ref{injectivehomomorphism} are the essential tools that we need in order to study the Wadge order on the Scott domain. 

This game comes as a standard two-player infinite game where the players choose elements of $\mathbb{P}$.

\begin{definition}\label{posetgame}
Let $\p,\q\in \mathbb{P}.$ The game $G_{\mathbb{P}}(\p,\q)$ is defined as a two-player ($\I$ and $\II$) game played on $\omega$ rounds. Each round $n\in\omega$ is played as follows: first $\I$ picks an element $p_n\in P$ and then $\II$ picks an element $q_n\in Q.$ We further require that there exists $n_0\in\omega$ such that, for all $n\geq n_0,$ $p_n=p_{n_0}.$

We say that $\II$ \emph{wins the game} if and only if the two following conditions are satisfied:
\begin{enumerate}[nolistsep]
\item $p_n\leq_p p_m\to q_n\leq_q q_m$ holds for all $n,m\in\omega$,
\item $\col_p(p_n)=\col_q(q_n)$ for all $n\in\omega$.
\end{enumerate}

Schematically, the game goes as follows:

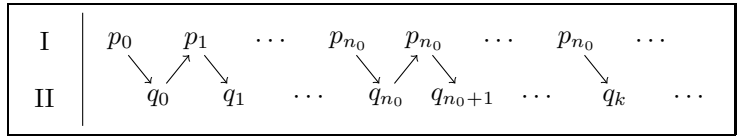
\begin{figure}[H]
\centering
\fbox{
\begin{tikzpicture}[yscale=0.75]
\draw (0,1) node {$\I$};
\draw (0,0) node {$\II$};
\draw (0.5,-0.5) -- (0.5,1.5);
\draw (1,1) node {$p_0$};
\draw (1.5,0) node {$q_0$};
\draw (2,1) node {$p_1$};
\draw (2.5,0) node {$q_1$};
\draw (3,1) node {$\cdots$};
\draw (3.5,0) node {$\cdots$};
\draw (4,1) node {$p_{n_0}$};
\draw (4.5,0) node {$q_{n_0}$};
\draw (5,1) node {$p_{n_0}$};
\draw (5.5,0) node {$q_{n_0+1}$};
\draw (6,1) node {$\cdots$};
\draw (6.5,0) node {$\cdots$};
\draw (7,1) node {$p_{n_0}$};
\draw (7.5,0) node {$q_k$};
\draw (8,1) node {$\cdots$};
\draw (8.5,0) node {$\cdots$};
\draw [->] (1.1,0.75) -- (1.4,0.25);
\draw [->] (2.1,0.75) -- (2.4,0.25);
\draw [->] (4.1,0.75) -- (4.4,0.25);
\draw [->] (5.1,0.75) -- (5.4,0.25);
\draw [->] (7.1,0.75) -- (7.4,0.25);
\draw [->] (1.6,0.25) -- (1.9,0.75);
\draw [->] (4.6,0.25) -- (4.9,0.75);
%\draw [->] (0.25,4.9) -- (0.75,4.6);
%\draw [->] (0.25,3.9) -- (0.75,3.6);
%\draw [->] (0.25,2.9) -- (0.75,2.6);
%\draw [->] (0.25,1.9) -- (0.75,1.6);
%\draw [->] (0.25,0.9) -- (0.75,0.6);
%\draw [->] (0.25,-0.1) -- (0.75,-0.4);
%\draw [->] (0.25,-1.1) -- (0.75,-1.4);
%\draw [<-] (0.25,4.1) -- (0.75,4.4);
%\draw [<-] (0.25,3.1) -- (0.75,3.4);
%\draw [<-] (0.25,2.1) -- (0.75,2.4);
%\draw [<-] (0.25,1.1) -- (0.75,1.4);
%\draw [<-] (0.25,0.1) -- (0.75,0.4);
%\draw [<-] (0.25,-1.9) -- (0.75,-1.6);
%\draw node at (4,-1) {\big($p_i\in P$ and $q_i\in Q$ for all $i\in\w$\big)};
\end{tikzpicture}
}
\caption{The game $G_\mathbb{P}(P,Q)$ for $\p,\q\in \mathbb{P}$.}
\end{figure}

A \emph{run} of the game is a sequence $(p_0,q_0,p_1,q_1,\dots)\in (P\cup Q)^\omega.$
\end{definition}

In plain English, player $\I$ moves inside the 2-colored poset $\p$, whereas player $\II$ moves inside the 2-colored poset $\q$. The goal of $\II$ is to reproduce (order-wise and color-wise) in $\q$ the run that $\I$ is producing in $\p.$ Notice that the condition of playing ultimately constant for player $\I$ is equivalent to requiring that the game stops after finitely many rounds.

Related to this game, we introduce the notion of an \emph{ultrapositional strategy} as a strengthening of the usual notion of a strategy.

\begin{definition}
Let $\p,\q\in\mathbb{P}$. An ultrapositional strategy for player $\II$ in the game $G_{\mathbb{P}}(\p,\q)$ is a function $\tau:P\to Q.$
\end{definition}

Contrary to the usual strategies that rely on the history of the opponent's run, ultrapositional strategies only take into account the last move of the opponent. An ultrapositional strategy is winning if it ensures a win whatever the opponent does.

Ultrapositional strategies characterize the reductions inside $\mathbb{P}$ as shown by the next proposition.

\begin{proposition} \label{propcharposet}
Let $\p,\q\in \mathbb{P}.$ 
$$\p\preccurlyeq_c \q \Longleftrightarrow \text{ $\II$ has an ultrapositional winning strategy in $G_{\mathbb{P}}(\p,\q).$}$$
\end{proposition}

\begin{proof}
First, suppose that $\p\preccurlyeq_c \q$ holds and is witnessed by $\varphi:P\to Q$. Observe that $\varphi$ is also an ultrapositional strategy. From the very definition of a homomorphism between 2-colored posets, it respects the two conditions to be winning for $\II$ in $G_{\mathbb{P}}(\p,\q).$

Conversely, an ultrapositional winning strategy for $\II$ in $G_{\mathbb{P}}(\p,\q)$ is a homomorphism $\varphi:P\to Q$ for it respects the two winning conditions.
\end{proof}

We obtain a reduction between 2-colored posets and their subposets that are closed under the predecessor relation.

\begin{definition}
Let $(Q,\leq_q)$ be a poset. A subposet $(P,\leq_p)$ is an \emph{ideal} of $(Q,\leq_q)$ if, for all $p\in P,$ we have $\{q\in Q:q\leq_q p\}\subseteq P.$
\end{definition}

\begin{proposition}\label{inducedprefixsubposet}
Let $\p,\q\in \mathbb{P}.$
\begin{center}
If $\p$ is an ideal of $\q,$ then $\p\preccurlyeq_c \q.$
\end{center}
\end{proposition}

\begin{proof}
The inclusion $i: P\to Q,\ p\mapsto p$ witnessing that $(P,\leq_p)$ is an ideal of $(Q,\leq_q)$ is an ultrapositional winning strategy for $\II$ in $G_{\mathbb{P}}(\p,\q)$.
\end{proof}

\subsection{On the reduction game on $\boldsymbol{\mathbb{P}_{\fin}}$}

In order to simplify some later proofs, we conclude this section with some necessary conditions for an ultrapositional strategy to be winning in a subclass of the embeddable posets. 

\begin{definition}
A finite branching poset is an embeddable poset $\p\in\mathbb{P}_{\lay}$ such that every element $p\in P$ which is not $\leq_p$-minimal has finitely many successors, i.e., for all $p\in P,$ if $p\neq \bot,$ then: $$\Card\big(\Succ(p)\big)=\Card\big(\{p'\in P \mid p\leq_p p'\}\big)<\aleph_0.$$ The class of all finite branching posets is denoted by $\mathbb{P}_{\fin}$.
\end{definition}

It turns out that the image of a finitely branching poset via the order-embedding of Theorem \ref{injectivehomomorphism} must be topologically reasonably simple, for we have:

\begin{proposition}\label{orderembedding}
If $\p\in \mathbb{P}_{\fin}$, then $\A_\p\in D_\w(\bsigma^0_1)(\scott).$
\end{proposition}

\begin{proof}
We use the characterization of Corollary \ref{chardiff}. Since $\p\in \mathbb{P}_{\reg}$ holds, Lemma \ref{lemmainjectivehomomorphism} implies that $\A_{\p}\in \bdelta^0_2(\scott)$ holds as well. Towards a contradiction, assume that $\A_{\p}$ admits a 1-alternating tree of rank $\omega$, namely: $$f:T_\omega\to \PP_{<\omega}(\omega).$$ This implies that, for every $k\in\w,$ there exists a strictly $\subseteq$-increasing sequence $(F^k_m)_{m<k}$ such that $F^k_0=f(\varnothing)$ and $F^k_m\in\A_{\p}$ both hold for all $m<k.$ Thus, the sequence $\Big(l^{-1}_{p}\big(F^k_m\big)\Big)_{l<k}$ is a strictly $\leq_{p}$-increasing sequence of size $k$ that satisfies $$\col_{p}\Big(l^{-1}_{p}\big(F^k_0\big)\Big)=\col_{p}\Big(l^{-1}_{p}\big(f(\varnothing)\big)\Big)=1,$$ for every $k\in\w.$ Therefore, we obtain $$\Card\bigg(\Succ\Big(l^{-1}_{p}\big(f(\varnothing)\big)\Big)\bigg)=\aleph_0.$$ By definition of a finite branching poset, this implies $l^{-1}_{p}\big(f(\varnothing)\big)=\bot,$ a contradiction for $\col_p(\bot)=0.$
\end{proof}

As a corollary, we obtain a somehow more detailed picture of Theorem \ref{injectivehomomorphism}.

\begin{corollary}
The following mapping is an order-embedding:
\begin{align*}
H:\sfrac{({\mathbb{P}_{\fin}},\preccurlyeq_c)}{\equiv_c} &\to \sfrac{(D_\w(\bsigma^0_1)(\Pw),\leq_w)}{\equiv_w}\\ [\p]&\mapsto [\A_\p].
\end{align*}

\end{corollary}

Now, we introduce some notations to talk about the game-theoretical strength of a given node in a finite branching poset.

Let us fix $\p\in \mathbb{P}_{\fin}$ and $p\in P.$ If it exists, let $k_p\in\omega$ be the length of the largest strictly $\leq_p$-increasing sequence $(s_n)_{n<k_p}$ that satisfies $s_0=p$ and $(\col_p(s_n)=\col_p(p)\Leftrightarrow \text{$n$ is even}).$ The \emph{increasing strength of $p$ in $\p$} is $$\Str_{\incr}(p)=\begin{cases} k_p &\text{ if $k_p\in\w$ exists,} \\ \omega &\text{ otherwise.}\end{cases}$$ Since $\p\in\mathbb{P}_{\fin}$, the latter case can only occur when $p=\bot$. From a game-theoretical viewpoint, if $p\neq \bot$, then $\Str_{\incr}(p)$ corresponds to the length of the strongest $<_p$-increasing run that a player can take while playing in $\p.$

In a similar manner, we define the \emph{decreasing strength of $p$ in $\p$}, denoted by $\Str_{\decr}(p)=k\in\omega,$ as the length of the largest strictly $\leq_p$-decreasing sequence $(s_n)_{n<k}$ that satisfies $s_0=p$ and $(\col_p(s_n)=\col_p(p)\Leftrightarrow \text{$n$ is even}).$ It is well-defined since $\Card(\Pred(p))<\aleph_0$ holds for every $p\in P.$

The increasing and decreasing strengths of a node give a good indication of the strength it bears as a position in the game:

\begin{lemma} \label{strength}
If $\p,\q\in \mathbb{P}_{\fin}$ and $\tau$ is a winning ultrapositional strategy for $\II$ in the game $G_{\mathbb{P}}(\p,\q),$ then for all $p\in P:$ 
\begin{enumerate}[nolistsep]
\item $\Str_{\incr}(p) \leq \Str_{\incr}\big(\tau(p)\big),$ 
\item $\Str_{\decr}(p) \leq \Str_{\decr}\big(\tau(p)\big).$\end{enumerate}
\end{lemma}

\begin{proof} ~
\begin{enumerate}
\item Towards a contradiction, suppose that $\Str_{\incr}(p) > \Str_{\incr}\big(\tau(p)\big)$. We proceed by cases. 
\begin{description}
\item[If $\boldsymbol{\Str_{\incr}(p)\neq \w}$:] assume that $\Str_{\incr}(p)=k$ is witnessed by a sequence $(p_n)_{n<k}.$ Since $\tau$ is winning, $\big(\tau(p_n)\big)_{n<k}$ is strictly $\leq_q$-increasing and satisfies $\tau(p_0)=\tau(p)$ and $\big(\col_q\big(\tau(p_n)\big)=\col_p\big(\tau(p)\big)\Leftrightarrow \text{$n$ is even}\big).$ Thus $\Str_{\incr}\big(\tau(p)\big)\geq k,$ a contradiction. 
\item[If $\boldsymbol{\Str_{\incr}(p)= \w}$:] for all $k\in \w,$ there exists a strictly $\leq_p$-increasing sequence $(s_n)_{n<k}$ that satisfies $s_0=p$ and $(\col_p(s_n)=\col_p(p)\Leftrightarrow \text{$n$ is even})$. Since $\tau$ is winning, $\big(\tau(p_n)\big)_{n<k}$ is strictly $\leq_q$-increasing and satisfies $\tau(p_0)=\tau(p)$ and $\big(\col_q\big(\tau(p_n)\big)=\col_p\big(\tau(p)\big)\Leftrightarrow \text{$n$ is even}\big).$ Therefore, $\Str_{\incr}\big(\tau(p)\big)=\w,$ a contradiction.
\end{description}

\item Towards a contradiction, suppose that $\Str_{\decr}(p) > \Str_{\decr}\big(\tau(p)\big)$. We also suppose that $\Str_{\decr}(p)=k\in\w$ is witnessed by a sequence $(p_n)_{n<k}.$ Since $\tau$ is winning, $\big(\tau(p_n)\big)_{n<k}$ is strictly $\leq_q$-decreasing and satisfies $\tau(p_0)=\tau(p)$ and $\big(\col_q\big(\tau(p_n)\big)=\col_p\big(\tau(p)\big)\Leftrightarrow \text{$n$ is even}\big).$ Thus $\Str_{\decr}\big(\tau(p)\big)\geq k,$ a contradiction.
\end{enumerate}
\end{proof}

\section{Ill-foundedness of the Wadge order on the Scott domain}

In this section, we prove that the quasi-order $\leq_w$ is already ill-founded inside the class of $\w$-differences of open sets of the Scott domain.

\begin{theorem}\label{illfounded} ~
\begin{center}
$\big(D_\omega(\bsigma^0_1)(\Pw),\leq_w\!\!\big)$ is ill-founded.
\end{center}
\end{theorem}

\begin{proof}
The proof consists in exhibiting a strictly $\preccurlyeq_c$-decreasing sequence of posets $(\p_n)_{n\in\w^+}$ in $\mathbb{P}_{\lay}$ and making use of the Lemma \ref{lemmainjectivehomomorphism}.

First, let us fix $n\in\w^+.$ We define $\p_n=(P_n,\leq_{p_n},\col_{p_n})$ as the following 2-colored countable poset:

\begin{figure}[H]
\centering
\fbox{
\scalebox{0.75}{
\begin{tikzpicture}[node distance=1.2cm and 1cm]
\node (wn) [fill=lightgray,rounded corners] at (0,0) {$w_{n}$};
\node (wn-1) [left of = wn,fill=lightgray,rounded corners] {$w_{n-1}$};
\node (wn-2) [left of = wn-1] {$\cdots$};
\node (w2) [left of = wn-2,fill=lightgray,rounded corners] {$w_2$};
\node (w1) [left of = w2,fill=lightgray,rounded corners] {$w_1$};
\node (w0) [left of = w1,fill=lightgray,rounded corners] {$w_0$};
\node (wn+1) [right of = wn,fill=lightgray,rounded corners] {$w_{n+1}$};
\node (wn+2) [right of = wn+1] {$\cdots$};
\node (w2n-1) [right of = wn+2,fill=lightgray,rounded corners] {$w_{2n-1}$};
\node (w2n) [right of = w2n-1,fill=lightgray,rounded corners] {$w_{2n}$};
\node (w2n+1) [right of = w2n,fill=lightgray,rounded corners] {$w_{2n+1}$};
\node (w2n+2) [right of = w2n+1] {$\cdots$};

\node (bot) at (0,-2) {$\bot$};

\node (xn) at (0,1) {$x_{n}$};
\node (xn-1) [left of = xn] {$x_{n-1}$};
\node (xn-2) [left of = xn-1] {$\cdots$};
\node (x2) [left of = xn-2] {$x_2$};
\node (x1) [left of = x2] {$x_1$};
\node (x0) [left of = x1] {$x_0$};
\node (xn+1) [right of = xn] {$x_{n+1}$};
\node (xn+2) [right of = xn+1] {$\cdots$};
\node (x2n-1) [right of = xn+2] {$x_{2n-1}$};
\node (x2n) [right of = x2n-1] {$x_{2n}$};
\node (x2n+1) [right of = x2n] {$x_{2n+1}$};
\node (x2n+2) [right of = x2n+1] {$\cdots$};

\node (yn) at (0.5,2) {$y_{n}$};
\node (yn-1) [left of = yn] {$y_{n-1}$};
\node (yn-2) [left of = yn-1] {$\cdots$};
\node (y2) [left of = yn-2] {$y_2$};
\node (y1) [left of = y2] {$y_1$};
\node (y0) [left of = y1] {$y_0$};
\node (yn+1) [right of = yn] {$y_{n+1}$};
\node (yn+2) [right of = yn+1] {$\cdots$};
\node (y2n-1) [right of = yn+2] {$y_{2n-1}$};
\node (y2n) [right of = y2n-1] {$y_{2n}$};
\node (y2n+1) [right of = y2n] {$y_{2n+1}$};
\node (y2n+2) [right of = y2n+1] {$\cdots$};

\node (zn) [fill=lightgray,rounded corners] at (0.5,3) {$z^0_{n}$};
\node (zn-1) [left of = zn,fill=lightgray,rounded corners] {$z^0_{n-1}$};
\node (zn-2) [left of = zn-1] {$\cdots$};
\node (z2) [left of = zn-2,fill=lightgray,rounded corners] {$z^0_2$};
\node (z1) [left of = z2,fill=lightgray,rounded corners] {$z^0_1$};
\node (z0) [left of = z1,fill=lightgray,rounded corners] {$z^0_0$};
\node (zn+1) [right of = zn,fill=lightgray,rounded corners] {$z^0_{n+1}$};
\node (zn+2) [right of = zn+1] {$\cdots$};
\node (z2n-1) [right of = zn+2,fill=lightgray,rounded corners] {$z^0_{2n-1}$};
\node (z2n) [right of = z2n-1,fill=lightgray,rounded corners] {$z^0_{2n}$};
\node (z2n+1) [right of = z2n,fill=lightgray,rounded corners] {$z^0_{2n+1}$};
\node (z2n+2) [right of = z2n+1] {$\cdots$};

\node (vn) at (0.5,4) {$z^1_{n}$};
\node (vn+1) [right of = vn] {$z^1_{n+1}$};
\node (vn+2) [right of = vn+1] {$\cdots$};
\node (v2n-1) [right of = vn+2] {$z^1_{2n-1}$};
\node (v2n) [right of = v2n-1] {$z^1_{2n}$};
\node (v2n+1) [right of = v2n] {$z^1_{2n+1}$};
\node (v2n+2) [right of = v2n+1] {$\cdots$};

\node (un) [fill=lightgray,rounded corners] at (0.5,5) {$z^2_{n}$};
\node (un+1) [right of = un,fill=lightgray,rounded corners] {$z^2_{n+1}$};
\node (un+2) [right of = un+1] {$\cdots$};
\node (u2n-1) [right of = un+2,fill=lightgray,rounded corners] {$z^2_{2n-1}$};
\node (u2n) [right of = u2n-1,fill=lightgray,rounded corners] {$z^2_{2n}$};
\node (u2n+1) [right of = u2n,fill=lightgray,rounded corners] {$z^2_{2n+1}$};
\node (u2n+2) [right of = u2n+1] {$\cdots$};

\node (t2n) at (5.3,6) {$z^3_{2n}$};
\node (t2n+1) [right of = t2n] {$z^3_{2n+1}$};
\node (t2n+2) [right of = t2n+1] {$\cdots$};

\node (s2n) [fill=lightgray,rounded corners] at (5.3,7) {$z^4_{2n}$};
\node (s2n+1) [right of = s2n,fill=lightgray,rounded corners] {$z^4_{2n+1}$};
\node (s2n+2) [right of = s2n+1] {$\cdots$};

\draw [->,thick] (bot) -- (w0);
\draw [->,thick] (bot) -- (w1);
\draw [->,thick] (bot) -- (w2);
\draw [->,thick] (bot) -- (wn-1);
\draw [->,thick] (bot) -- (wn);
\draw [->,thick] (bot) -- (wn+1);
\draw [->,thick] (bot) -- (w2n-1);
\draw [->,thick] (bot) -- (w2n);
\draw [->,thick] (bot) -- (w2n+1);

\draw [->,thick] (w0) -- (x0);
\draw [->,thick] (w1) -- (x1);
\draw [->,thick] (w2) -- (x2);
\draw [->,thick] (wn-1) -- (xn-1);
\draw [->,thick] (wn) -- (xn);
\draw [->,thick] (wn+1) -- (xn+1);
\draw [->,thick] (w2n-1) -- (x2n-1);
\draw [->,thick] (w2n) -- (x2n);
\draw [->,thick] (w2n+1) -- (x2n+1);

\draw [->,thick] (x0) -- (y0);
\draw [->,thick] (x1) -- (y1);
\draw [->,thick] (x2) -- (y2);
\draw [->,thick] (xn-1) -- (yn-1);
\draw [->,thick] (xn) -- (yn);
\draw [->,thick] (xn+1) -- (yn+1);
\draw [->,thick] (x2n-1) -- (y2n-1);
\draw [->,thick] (x2n) -- (y2n);
\draw [->,thick] (x2n+1) -- (y2n+1);

\draw [->,thick] (x1) -- (y0);
\draw [->,thick] (x2) -- (y1);
\draw [->,thick] (xn) -- (yn-1);
\draw [->,thick] (xn+1) -- (yn);
\draw [->,thick] (x2n) -- (y2n-1);
\draw [->,thick] (x2n+1) -- (y2n);

\draw [->,thick] (y0) -- (z0);
\draw [->,thick] (y1) -- (z1);
\draw [->,thick] (y2) -- (z2);
\draw [->,thick] (yn-1) -- (zn-1);
\draw [->,thick] (yn) -- (zn);
\draw [->,thick] (yn+1) -- (zn+1);
\draw [->,thick] (y2n-1) -- (z2n-1);
\draw [->,thick] (y2n) -- (z2n);
\draw [->,thick] (y2n+1) -- (z2n+1);

\draw [->,thick] (zn) -- (vn);
\draw [->,thick] (zn+1) -- (vn+1);
\draw [->,thick] (z2n-1) -- (v2n-1);
\draw [->,thick] (z2n) -- (v2n);
\draw [->,thick] (z2n+1) -- (v2n+1);

\draw [->,thick] (vn) -- (un);
\draw [->,thick] (vn+1) -- (un+1);
\draw [->,thick] (v2n-1) -- (u2n-1);
\draw [->,thick] (v2n) -- (u2n);
\draw [->,thick] (v2n+1) -- (u2n+1);

\draw [->,thick] (u2n) -- (t2n);
\draw [->,thick] (u2n+1) -- (t2n+1);

\draw [->,thick] (t2n) -- (s2n);
\draw [->,thick] (t2n+1) -- (s2n+1);
\end{tikzpicture}}
}
\caption{The 2-colored countable poset $\p_n\in \mathbb{P}_{\lay}$ for $n\in\w^+$.}\label{posetpn}
\end{figure}

Formally, the set of nodes is:
\begin{align*}
P_n&=\{\bot\}\cup \{w_m,x_m,y_m\}_{m\in\omega} \\ 
&\ \ \ \cup\big\{z^{2k}_{m}\mid k\in\omega, n\geq km\big\} \cup\big\{z^{2k+1}_{m}\mid k\in\omega, n\geq (k+1)m\big\},
\end{align*}
the order relation is:
\begin{align*}
\leq_{p_n}&=\big\{(\bot,w_m),(w_m,x_m),(x_m,y_m),(x_{m+1},y_m),(y_m,z^0_{m})\big\}_{m\in\omega}  \\ 
&\ \ \ \cup \left\{(z^k_{m},z^{k+1}_{m})\mid k\leq \left\lfloor\frac{m}{n}\right\rfloor\cdot 2-1\right\},
\end{align*}
where $\left\lfloor\frac{m}{n}\right\rfloor$ denotes the integer part of $\frac{m}{n},$ and the 2-coloring is:
\begin{align*}
\col_{p_n}: P_n &\to 2 \\
p&\mapsto 0 \ \ \  \text{ if $p\in \{\bot, x_m, y_m\}_{m\in\w} \cup \bigcup_{m\in\w} z^{\odd}_m$}, \\
p&\mapsto 1 \ \ \  \text{ if $p\in \{w_m\}_{m\in\w} \cup \bigcup_{m\in\w} z^{\even}_m$}, 
\end{align*}
where $z^\cdot_m=\{z^k_m\mid k\leq\left\lfloor\frac{m}{n}\right\rfloor\cdot 2\},$ $z^{\even}_m=\{z^k_m\in z^\cdot_k\mid \text{$k$ even}\},$ and $z^{\odd}_m=\{z^k_m\in z^\cdot_k\mid \text{$k$ odd}\}.$

For all $n\in\w^+,$ it is easy to check that all the requirements that are needed for $\p_n$ to belong to $\mathbb{P}_{\fin}$ are fulfilled. Therefore, by Proposition \ref{orderembedding}, we have: $$\A_{\p_n}\in D_\omega(\bsigma^0_1)(\scott).$$

For the remainder of the proof, we need some notations. For any $k\in\w,$ we call \emph{branch $k$ of $\p_n$} the set of nodes $B_k=\{w_k,x_k,y_k\}\cup z^\cdot_k,$ and \emph{right-shift in $\p_n$} any sequence of moves of the form $(w_k,y_k,w_{k+1}).$ First, we describe the behavior of an ultrapositional winning strategy facing a right-shift.

\begin{claim} \label{claim35}
Let $n,m\in\w^+$ and $\tau$ be an ultrapositional strategy in $G_{\mathbb{P}}(\p_n,\p_m)$. If $\I$'s moves are a right-shift $(w_k,y_k,w_{k+1})$ and $\tau(w_k)\in B_l$ for some $l\in\w,$ then $\tau(w_{k+1})\in B_{l'}$ for some $l'\leq l+1$.
\end{claim}

\begin{claimproof}
We split the proof in two different cases.
\begin{description}
\item[If $\boldsymbol{l=0}$ holds:]
since $w_k\leq_{p_n} y_k,$ $\col_{p_n}(y_k)=0$, $\tau$ is winning and $\tau(w_k)\in B_0$, we get $\tau(y_k)\in \{x_0,y_0\}.$ Moreover, since $w_{k+1}\leq_{p_n} y_k,$ $\col_{p_n}(w_{k+1})=1$ and $\tau$ is winning, we get:
$$\tau(w_{k+1})\in \{w_0,w_{1}\} \subseteq B_{0}\cup B_{1}.$$
\item[If $\boldsymbol{l\in\w^+}$ holds:]
once again, since $w_k\leq_{p_n} y_k,$ $\col_{p_n}(y_k)=0$, $\tau$ is winning and $\tau(w_k)\in B_l$, we get $\tau(y_k)\in z^{\odd}_{l-1}\cup z^{\odd}_{l} \cup \{x_l,y_l,y_{l-1}\}.$ Moreover, since $w_{k+1}\leq_{p_n} y_k,$ $\col_{p_n}(w_{k+1})=1$ and $\tau$ is winning, we get:
$$\tau(w_{k+1})\in z^{\even}_{l-1}\cup z^{\even}_{l} \cup \{w_{l-1},w_l,w_{l+1}\} \subseteq \bigcup_{l'\leq l+1} B_{l'}.$$
\end{description}
\end{claimproof}

It remains to show that the sequence $(\p_n)_{n\in\w_+}$ is an infinite strictly $\preccurlyeq_c$-decreasing sequence in $\mathbb{P}_{\lay}$.

\begin{claim}
If $0<n<m<\w,$ then $\p_m\preccurlyeq_c \p_n.$
\end{claim}

\begin{claimproof}
It suffices to observe that $\p_m$ is an ideal of $\p_n$ and use Proposition \ref{inducedprefixsubposet}.
\end{claimproof}

\begin{claim}\label{claim37}
If $0<n<m<\w,$ then $\p_n\not\preccurlyeq_c \p_m.$
\end{claim}

\begin{claimproof}
Towards a contradiction, suppose that $\p_n\preccurlyeq_c \p_m$ holds. By Proposition \ref{propcharposet}, player $\II$ has a winning ultrapositional strategy $\tau$ in the game $G_{\mathbb{P}}(\p_n,\p_m)$.

The idea of the proof is to construct a particular run of the game that $\tau$ cannot win. By Claim \ref{claim35}, if $\I$ plays a sequence of the form $(w_0,y_0,w_1,y_1,w_2,\dots)$ composed with right-shifts, then $\II$'s moves are limited. In particular, whenever $\I$ shifts from $B_k$ to $B_{k+1},$ $\II$ can only shift from $B_l$ to $B_{l'}$ where $l'\leq l+1.$ Because $n<m,$ $\I$ can finally reach a node of greater increasing strength than the one reached by $\II,$ which leads to a contradiction.

More formally, suppose that $\I$'s first move is $w_0$ so that $\tau(w_0)\in B_{k_0}$ for some $k_0\in\w,$ and that $\I$ plays a run composed with several right-shifts $$(w_0,y_0,w_1,y_1,w_2,\dots,w_l).$$ By an iteration of Claim \ref{claim35}, we get $\tau(w_l)\in B_{l'}$ for some $l'\leq k_0+l$. Since $n<m,$ there exists $n_0\in \w$ such that the following inequalities work:
$$\Str_{\incr}(w_{nmn_0})=2mn_0+3>2nn_0+\Str_{\incr}(w_{k_0})\geq\Str_{\incr}\big(\tau(w_{nmn_0})\big),$$
which is a contradiction to Lemma \ref{strength}.
\end{claimproof}

So, we constructed an infinite strictly $\preccurlyeq_c$-decreasing sequence of embeddable posets, namely $$\p_1\succ_c \p_2 \succ_c \p_3 \succ_c \p_4 \succ_c \dots$$
By Lemma \ref{lemmainjectivehomomorphism}, we obtain an infinite strictly $\leq_w$-decreasing sequence of subsets of $\scott$, namely: $$\A_{\p_1}
>_w\A_{\p_2}>_w\A_{\p_3}>_w\A_{\p_4}>_w\dots$$
which were also proved to be differences of $\w$ open sets.
\end{proof}

\section{Antichains in the Wadge order on the Scott domain}

We prove that infinite $\leq_w$-antichains already exist within the class of $\w$-differences of open subsets of the Scott domain. The proof is nothing but a tailoring of the proof of Theorem \ref{illfounded}.

\begin{theorem}\label{antichain} ~
\begin{center}$\big(D_\omega(\bsigma^0_1)(\Pw),\leq_w)$ has infinite antichains.\end{center}
\end{theorem}

\begin{proof}
We construct an infinite sequence of embeddable posets $(\q_n)_{n\in\w_+}$ that are pairwise $\preccurlyeq_c$-incomparable.

We fix $n\in \omega^+$ and define $\q_n=(Q_n, \leq_{q_n},\col_{q_n})$ as the following 2-colored countable poset:

\begin{figure}[H]
\centering
\fbox{
\scalebox{0.75}{
\begin{tikzpicture}[scale=1,node distance=1.2cm and 1cm]
\node (wn) [fill=lightgray,rounded corners] at (0,-3) {$x^0_{n}$};
\node (wn-1) [left of = wn,fill=lightgray,rounded corners] {$x^0_{n-1}$};
\node (wn-2) [left of = wn-1] {$\cdots$};
\node (w2) [left of = wn-2,fill=lightgray,rounded corners] {$x^0_2$};
\node (w1) [left of = w2,fill=lightgray,rounded corners] {$x^0_1$};
\node (w0) [left of = w1,fill=lightgray,rounded corners] {$x^0_0$};
\node (wn+1) [right of = wn,fill=lightgray,rounded corners] {$x^0_{n+1}$};
\node (wn+2) [right of = wn+1] {$\cdots$};
\node (w2n-1) [right of = wn+2,fill=lightgray,rounded corners] {$x^0_{2n-1}$};
\node (w2n) [right of = w2n-1,fill=lightgray,rounded corners] {$x^0_{2n}$};
\node (w2n+1) [right of = w2n,fill=lightgray,rounded corners] {$x^0_{2n+1}$};
\node (w2n+2) [right of = w2n+1] {$\cdots$};

\node (an) at (0,-2) {$x^1_{n}$};
\node (an-1) [left of = an] {$x^1_{n-1}$};
\node (an-2) [left of = an-1] {$\cdots$};
\node (a2) [left of = an-2] {$x^1_2$};
\node (a1) [left of = a2] {$x^1_1$};
\node (a0) [left of = a1] {$x^1_0$};
\node (an+1) [right of = an] {$x^1_{n+1}$};
\node (an+2) [right of = an+1] {$\cdots$};
\node (a2n-1) [right of = an+2] {$x^1_{2n-1}$};
\node (a2n) [right of = a2n-1] {$x^1_{2n}$};
\node (a2n+1) [right of = a2n] {$x^1_{2n+1}$};
\node (a2n+2) [right of = a2n+1] {$\cdots$};

\node (bn) at (0,-1) {$\vdots$};
\node (bn-1) [left of = bn] {$\vdots$};
\node (bn-2) [left of = bn-1] {$\cdots$};
\node (b2) [left of = bn-2] {$\vdots$};
\node (b1) [left of = b2] {$\vdots$};
\node (b0) [left of = b1] {$\vdots$};
\node (bn+1) [right of = bn] {$\vdots$};
\node (bn+2) [right of = bn+1] {$\cdots$};
\node (b2n-1) [right of = bn+2] {$\vdots$};
\node (b2n) [right of = b2n-1] {$\vdots$};
\node (b2n+1) [right of = b2n] {$\vdots$};
\node (b2n+2) [right of = b2n+1] {$\cdots$};

\node (cn) [fill=lightgray,rounded corners] at (0,0) {$x^{2n-2}_{n}$};
\node (cn-1) [left of = cn,fill=lightgray,rounded corners] {$x^{2n-2}_{n-1}$};
\node (cn-2) [left of = cn-1] {$\cdots$};
\node (c2) [left of = cn-2,fill=lightgray,rounded corners] {$x^{2n-2}_2$};
\node (c1) [left of = c2,fill=lightgray,rounded corners] {$x^{2n-2}_1$};
\node (c0) [left of = c1,fill=lightgray,rounded corners] {$x^{2n-2}_0$};
\node (cn+1) [right of = cn,fill=lightgray,rounded corners] {$x^{2n-2}_{n+1}$};
\node (cn+2) [right of = cn+1] {$\cdots$};
\node (c2n-1) [right of = cn+2,fill=lightgray,rounded corners] {$x^{2n-2}_{2n-1}$};
\node (c2n) [right of = c2n-1,fill=lightgray,rounded corners] {$x^{2n-2}_{2n}$};
\node (c2n+1) [right of = c2n,fill=lightgray,rounded corners] {$x^{2n-2}_{2n+1}$};
\node (c2n+2) [right of = c2n+1] {$\cdots$};

\node (bot) at (0,-5) {$\bot$};

\node (xn) at (0,1) {$x^{2n-1}_{n}$};
\node (xn-1) [left of = xn] {$x^{2n-1}_{n-1}$};
\node (xn-2) [left of = xn-1] {$\cdots$};
\node (x2) [left of = xn-2] {$x^{2n-1}_2$};
\node (x1) [left of = x2] {$x^{2n-1}_1$};
\node (x0) [left of = x1] {$x^{2n-1}_0$};
\node (xn+1) [right of = xn] {$x^{2n-1}_{n+1}$};
\node (xn+2) [right of = xn+1] {$\cdots$};
\node (x2n-1) [right of = xn+2] {$x^{2n-1}_{2n-1}$};
\node (x2n) [right of = x2n-1] {$x^{2n-1}_{2n}$};
\node (x2n+1) [right of = x2n] {$x^{2n-1}_{2n+1}$};
\node (x2n+2) [right of = x2n+1] {$\cdots$};

\node (yn) at (0.6,2) {$y_{n}$};
\node (yn-1) [left of = yn] {$y_{n-1}$};
\node (yn-2) [left of = yn-1] {$\cdots$};
\node (y2) [left of = yn-2] {$y_2$};
\node (y1) [left of = y2] {$y_1$};
\node (y0) [left of = y1] {$y_0$};
\node (yn+1) [right of = yn] {$y_{n+1}$};
\node (yn+2) [right of = yn+1] {$\cdots$};
\node (y2n-1) [right of = yn+2] {$y_{2n-1}$};
\node (y2n) [right of = y2n-1] {$y_{2n}$};
\node (y2n+1) [right of = y2n] {$y_{2n+1}$};
\node (y2n+2) [right of = y2n+1] {$\cdots$};

\node (zn) [fill=lightgray,rounded corners] at (0.6,3) {$z^0_{n}$};
\node (zn-1) [left of = zn,fill=lightgray,rounded corners] {$z^0_{n-1}$};
\node (zn-2) [left of = zn-1] {$\cdots$};
\node (z2) [left of = zn-2,fill=lightgray,rounded corners] {$z^0_2$};
\node (z1) [left of = z2,fill=lightgray,rounded corners] {$z^0_1$};
\node (z0) [left of = z1,fill=lightgray,rounded corners] {$z^0_0$};
\node (zn+1) [right of = zn,fill=lightgray,rounded corners] {$z^0_{n+1}$};
\node (zn+2) [right of = zn+1] {$\cdots$};
\node (z2n-1) [right of = zn+2,fill=lightgray,rounded corners] {$z^0_{2n-1}$};
\node (z2n) [right of = z2n-1,fill=lightgray,rounded corners] {$z^0_{2n}$};
\node (z2n+1) [right of = z2n,fill=lightgray,rounded corners] {$z^0_{2n+1}$};
\node (z2n+2) [right of = z2n+1] {$\cdots$};

\node (vn) at (0.6,4) {$z^1_{n}$};
\node (vn+1) [right of = vn] {$z^1_{n+1}$};
\node (vn+2) [right of = vn+1] {$\cdots$};
\node (v2n-1) [right of = vn+2] {$z^1_{2n-1}$};
\node (v2n) [right of = v2n-1] {$z^1_{2n}$};
\node (v2n+1) [right of = v2n] {$z^1_{2n+1}$};
\node (v2n+2) [right of = v2n+1] {$\cdots$};

\node (un) [fill=lightgray,rounded corners] at (0.6,5) {$z^2_{n}$};
\node (un+1) [right of = un,fill=lightgray,rounded corners] {$z^2_{n+1}$};
\node (un+2) [right of = un+1] {$\cdots$};
\node (u2n-1) [right of = un+2,fill=lightgray,rounded corners] {$z^2_{2n-1}$};
\node (u2n) [right of = u2n-1,fill=lightgray,rounded corners] {$z^2_{2n}$};
\node (u2n+1) [right of = u2n,fill=lightgray,rounded corners] {$z^2_{2n+1}$};
\node (u2n+2) [right of = u2n+1] {$\cdots$};

\node (t2n) at (5.4,6) {$z^3_{2n}$};
\node (t2n+1) [right of = t2n] {$z^3_{2n+1}$};
\node (t2n+2) [right of = t2n+1] {$\cdots$};

\node (s2n) [fill=lightgray,rounded corners] at (5.4,7) {$z^4_{2n}$};
\node (s2n+1) [right of = s2n,fill=lightgray,rounded corners] {$z^4_{2n+1}$};
\node (s2n+2) [right of = s2n+1] {$\cdots$};

\draw [->,thick] (bot) -- (w0);
\draw [->,thick] (bot) -- (w1);
\draw [->,thick] (bot) -- (w2);
\draw [->,thick] (bot) -- (wn-1);
\draw [->,thick] (bot) -- (wn);
\draw [->,thick] (bot) -- (wn+1);
\draw [->,thick] (bot) -- (w2n-1);
\draw [->,thick] (bot) -- (w2n);
\draw [->,thick] (bot) -- (w2n+1);

\draw [->,thick] (w0) -- (a0);
\draw [->,thick] (w1) -- (a1);
\draw [->,thick] (w2) -- (a2);
\draw [->,thick] (wn-1) -- (an-1);
\draw [->,thick] (wn) -- (an);
\draw [->,thick] (wn+1) -- (an+1);
\draw [->,thick] (w2n-1) -- (a2n-1);
\draw [->,thick] (w2n) -- (a2n);
\draw [->,thick] (w2n+1) -- (a2n+1);

\draw [->,thick] (a0) -- (b0);
\draw [->,thick] (a1) -- (b1);
\draw [->,thick] (a2) -- (b2);
\draw [->,thick] (an-1) -- (bn-1);
\draw [->,thick] (an) -- (bn);
\draw [->,thick] (an+1) -- (bn+1);
\draw [->,thick] (a2n-1) -- (b2n-1);
\draw [->,thick] (a2n) -- (b2n);
\draw [->,thick] (a2n+1) -- (b2n+1);

\draw [->,thick] (b0) -- (c0);
\draw [->,thick] (b1) -- (c1);
\draw [->,thick] (b2) -- (c2);
\draw [->,thick] (bn-1) -- (cn-1);
\draw [->,thick] (bn) -- (cn);
\draw [->,thick] (bn+1) -- (cn+1);
\draw [->,thick] (b2n-1) -- (c2n-1);
\draw [->,thick] (b2n) -- (c2n);
\draw [->,thick] (b2n+1) -- (c2n+1);

\draw [->,thick] (c0) -- (x0);
\draw [->,thick] (c1) -- (x1);
\draw [->,thick] (c2) -- (x2);
\draw [->,thick] (cn-1) -- (xn-1);
\draw [->,thick] (cn) -- (xn);
\draw [->,thick] (cn+1) -- (xn+1);
\draw [->,thick] (c2n-1) -- (x2n-1);
\draw [->,thick] (c2n) -- (x2n);
\draw [->,thick] (c2n+1) -- (x2n+1);

\draw [->,thick] (x0) -- (y0);
\draw [->,thick] (x1) -- (y1);
\draw [->,thick] (x2) -- (y2);
\draw [->,thick] (xn-1) -- (yn-1);
\draw [->,thick] (xn) -- (yn);
\draw [->,thick] (xn+1) -- (yn+1);
\draw [->,thick] (x2n-1) -- (y2n-1);
\draw [->,thick] (x2n) -- (y2n);
\draw [->,thick] (x2n+1) -- (y2n+1);

\draw [->,thick] (x1) -- (y0);
\draw [->,thick] (x2) -- (y1);
\draw [->,thick] (xn) -- (yn-1);
\draw [->,thick] (xn+1) -- (yn);
\draw [->,thick] (x2n) -- (y2n-1);
\draw [->,thick] (x2n+1) -- (y2n);

\draw [->,thick] (y0) -- (z0);
\draw [->,thick] (y1) -- (z1);
\draw [->,thick] (y2) -- (z2);
\draw [->,thick] (yn-1) -- (zn-1);
\draw [->,thick] (yn) -- (zn);
\draw [->,thick] (yn+1) -- (zn+1);
\draw [->,thick] (y2n-1) -- (z2n-1);
\draw [->,thick] (y2n) -- (z2n);
\draw [->,thick] (y2n+1) -- (z2n+1);

\draw [->,thick] (zn) -- (vn);
\draw [->,thick] (zn+1) -- (vn+1);
\draw [->,thick] (z2n-1) -- (v2n-1);
\draw [->,thick] (z2n) -- (v2n);
\draw [->,thick] (z2n+1) -- (v2n+1);

\draw [->,thick] (vn) -- (un);
\draw [->,thick] (vn+1) -- (un+1);
\draw [->,thick] (v2n-1) -- (u2n-1);
\draw [->,thick] (v2n) -- (u2n);
\draw [->,thick] (v2n+1) -- (u2n+1);

\draw [->,thick] (u2n) -- (t2n);
\draw [->,thick] (u2n+1) -- (t2n+1);

\draw [->,thick] (t2n) -- (s2n);
\draw [->,thick] (t2n+1) -- (s2n+1);
\end{tikzpicture}}
}
\caption{The 2-colored countable poset $\q_n\in \mathbb{P}_{\lay}$ for $n\in\w^+$.}\label{posetqn}
\end{figure}

Formally, the set of nodes is:
\begin{align*}
Q_n&=\{\bot\}\cup \{x^{k}_m,y_m\}_{m\in\omega, k<2n} \\ 
&\ \ \ \cup\big\{z^{2k}_{m}\mid k\in\omega, n\geq km\big\} \cup\big\{z^{2k+1}_{m}\mid k\in\omega, n\geq (k+1)m\big\},
\end{align*}
the order relation is:
\begin{align*}
\leq_{q_n}&=\big\{(\bot,x^0_m),(x^k_m,x^{k+1}_m),(x^{2n-1}_m,y_m),(x^{2n-1}_{m+1},y_m),(y_m,z^0_{m})\big\}_{m\in\omega,k<2n-1}  \\ 
&\ \ \ \cup \left\{(z^k_{m},z^{k+1}_{m})\mid k\leq \left\lfloor\frac{m}{n}\right\rfloor\cdot 2-1\right\},
\end{align*}
and the coloring is given by the function:
\begin{align*}
\col_{p_n}: P_n &\to 2 \\
p&\mapsto 0 \ \ \  \text{ if $p\in \{\bot, x^{2k+1}_m, y_m\}_{m\in\w,k<n} \cup \bigcup_{m\in\w} z^{\odd}_m$}, \\
p&\mapsto 1 \ \ \  \text{ if $p\in \{2^{2k}_m\}_{m\in\w,k<n} \cup \bigcup_{m\in\w} z^{\even}_m$}.
\end{align*}

As in the proof of Theorem $\ref{illfounded},$ it is easy to see that $\q_n\in \mathbb{P}_{\fin}$, and thus $\A_{\q_n}\in D_\omega(\bsigma^0_1)(\scott)$ holds for every $n\in \w^+.$ Now, it remains to show that $(\q_n)_{n\in\w_+}$ is a sequence of pairwise $\preccurlyeq_c$-incomparable embeddable posets.  For this purpose, we define a \emph{right-shift in $\q_n$} as any sequence of moves of the form $(x^{2n-2}_k,y_k,x^{2n-2}_{k+1})$ for some $k\in\w$. 

\begin{claim}
If $0<n<m<\w,$ then $\q_m\not\preccurlyeq_c  \q_n.$
\end{claim}

\begin{claimproof}
Towards a contradiction, we assume that $\q_m\preccurlyeq_c \q_n$ holds. By Proposition \ref{propcharposet}, $\II$ has an ultrapositional winning strategy $\tau$ in the game $G_{\mathbb{P}}(\q_m,\q_n)$.

The idea of the proof is to exhibit some specific run for $\I$ in this game that $\tau$ cannot beat. For this purpose, player $\I$ will use the fact that $n<m$ and several right-shifts to reach an element $q\in\q_n$ which has a larger increasing strength than $\tau(q)$.

We consider $x_0^{2m-2}$ as $\I$'s first move. If $\II$'s first move is $x_i^{2j}$ for some $i\in\omega$ and $j<n,$ then $\Str_{\decr}\big(x_0^{2m-2}\big)=2m>2n \geq \Str_{\decr}\big(x_i^{2j}\big),$ which contradicts Lemma \ref{strength}. Since $\col_{q_m}(x_0^{2m-2})=1$, we can assume that  $\tau\big(x_0^{2m-2}\big)=z^{2k}_{l_0}$ for some $k,l_0\in\omega$. 

If $\I$'s second move is $y_0,$ then $\II$'s second move has color $0$. Hence, $\II$'s second move is of the form $z^{2k'+1}_{l_0}$ for some $k'\in \omega$.

Since $\Str_{\decr}\big(x_1^{2m-2}\big)=2m>2n \geq \Str_{\decr}\big(x_i^{2j}\big)$ for all $j<n,$ if $\I$'s third move is $x_1^{2m-2}$, then Lemma \ref{strength} implies that $\II$'s third move cannot be of the form $x_{i}^{2j}$ for some $i,j\in\w$. So, $\II$'s third move is of the form $z^{2k''}_{l_0}$ for some $k''\in\w.$ 

Now, consider the run where $\I$ plays right-shifts: $$\big(x_0^{2m-2},y_0,x_1^{2m-2},y_1,x_2^{2m-2},y_2,\dots\big).$$ By the previous observations, $\II$ will only play in $z^\cdot_{l_0}.$ But there exists $i_0\in\w$ such that $$\Str_{\incr}(y_{i_0})>\max \{\Str_{\incr}(q)\mid q\in z^\cdot_{l_0}\},$$ which contradicts Lemma \ref{strength}.
\end{claimproof}

For the last two claims, we need to introduce the notion of \emph{branches in $\q_n$}. For any $k\in\w,$ we call \emph{branch $k$ of $\q_n$} the set of nodes $B_k=\{x^l_k,y_k\}_{l<2n}\cup z^\cdot_k.$ The next claim, which concerns the 2-colored countable posets of the form $\q_n$ for some $n\in\w^+$, is a tailoring of Claim \ref{claim35}.

\begin{claim}\label{claim40}
Let $n,m\in\w^+$ and $\tau$ be an ultrapositional strategy in $G_{\mathbb{P}}(\q_n,\q_m)$. If $\I$'s moves are a right-shift $(x^{2n-2}_k,y_k,x^{2n-2}_{k+1})$ and $\tau\big(x^{2n-2}_k\big)\in B_l$ holds for some $l\in\w,$ then $\tau\big(x^{2n-2}_{k+1}\big)\in B_{l'}$ holds for some $l'\leq l+1$.
\end{claim}

\begin{claimproof}
We proceed as in the proof of Claim \ref{claim35}, except that the right-shift $(w_k,y_k,w_{k+1})$ in $\p_n$ is replaced by the right-shift $(x^{2n-2}_k,y_k,x^{2n-2}_{k+1})$ in $\q_n$.
\end{claimproof}

With the help of the previous claim, we finally obtain:

\begin{claim}
If $0<n<m<\w,$ then $\q_n\not\preccurlyeq_c \q_m.$
\end{claim}

\begin{claimproof}
We proceed as in the proof of Claim \ref{claim37}. Towards a contradiction, suppose that $\q_n\preccurlyeq_c \q_m$ holds. By Proposition \ref{propcharposet}, player $\II$ has a winning ultrapositional strategy $\tau$ in the game $G_{\mathbb{P}}(\q_n,\q_m)$.

Suppose that $\I$'s first move is $x^{2n-2}_0$ so that $\tau\big(x^{2n-2}_0\big)\in B_{k_0}$ for some $k_0\in\w,$ and that $\I$ plays a run composed with several right-shifts $$\big(x^{2n-2}_0,y_0,x^{2n-2}_1,y_1,x^{2n-2}_2\dots,x^{2n-2}_l).$$ By an iteration of Claim \ref{claim40}, we get $\tau\big(x^{2n-2}_l\big)\in B_{l'}$ for some $l'\leq k_0+l$. Since $n<m,$ there exists $n_0\in \w$ such that the following inequalities work:
$$\Str_{\incr}\big(x^{2n-2}_{nmn_0}\big)=2mn_0+3>2nn_0+\Str_{\incr}\big(x^0_{k_0}\big)\geq\Str_{\incr}\Big(\tau\big(x^{2n-2}_{nmn_0}\big)\Big),$$
which contradicts Lemma \ref{strength}.
\end{claimproof}

So, we constructed an infinite sequence of pairwise $\preccurlyeq_c$-incomparable embeddable posets, namely $(\q_n)_{n\in\omega^+}.$
By Lemma \ref{lemmainjectivehomomorphism}, we obtain an infinite sequence of pairwise $\leq_w$-incomparable subsets of $\scott$, namely $\big(\A_{\q_n}\big)_{n\in\omega^+}.$
We also proved that all these sets are $\w$-differences of open sets.
\end{proof}

\section{Open questions}

We conclude with some related open questions that may serve as guidelines for future work.

In Theorem \ref{injectivehomomorphism}, we exhibited a partial order on a class of 2-colored countable posets which embeds in the Wadge order on the $\bdelta^0_2$-degrees of $\scott$. It would be desirable to find a better description of this partial order, as it was recently done in \cite{Kihara2017} for the Baire space $\w^\w$ -- the space of infinite sequence of integers endowed with the product of the discrete topology. More precisely, they showed that the Wadge order on the Borel subsets of $\w^\w$ can be represented by countable joins of countable transfinite nests of well-founded trees labeled by 2. Although such a description seems to be out of reach for the whole Borel subsets, a reasonable question would be:

\begin{question}
Is there any standard order-theoretic structure which is isomorphic to $\sfrac{\big(\bdelta^0_2(\scott),\leq_w\!\!\big)}{\equiv_w}$?
\end{question}

We showed that some unwanted properties already occur at a very low topological complexity level in the Wadge order of $\scott.$ By looking at some reductions that are more general than the continuous ones, these bad behaviors may disappear. For example, Motto Ros, Schlicht and Selivanov consider the class of $\bsigma^0_\w$-functions $\F_0=\{f:\scott\to\scott : f^{-1}(\A)\in \bsigma^0_\w(\scott) \text{ for any }\A\in \bsigma^0_\w(\scott)\}$ \cite{Mottoros2015}. They show that $\leq_{\F_0}$\footnote{We write $\A\leq_{\F_0} \B$ if there exists $f\in \F_0$ such that $f^{-1}[\B]=\A$.} induces a well-quasi-order on the Borel subsets of $\scott$. Thus, the following question seems of interest:

\begin{question}
For which classes of functions $\F\subseteq\F_0$ containing the continuous ones is the induced order $\leq_\F$ on the Borel subsets of $\scott$ a well-quasi-order?
\end{question}

Another relevant question concerns the possibility of extending our results to some other quasi-Polish spaces. We essentially focused on $\scott$ because it is universal among them. Since we showed that $\scott$ is not well-behaved with respect to the Wadge order, one may ask where the well-behaved quasi-Polish spaces may be found.

\begin{question}
Is there a natural characterization of the quasi-Polish spaces whose Wadge order on the Borel subsets is a well-quasi-order?
\end{question}

In the metrizable setting, Schlicht proved that the Polish spaces for which $\leq_w$ is a well-quasi-order on the Borel subsets are exactly the zero-dimensional ones \cite{Schlicht2018}. It would be interesting to know whether this property somehow extends to the quasi-Polish spaces.

\bibliographystyle{amsalpha}
\bibliography{biblio1}

\end{document}

%% file: macros/packages.tex
\usepackage[english]{babel}
\usepackage[T1]{fontenc}
\usepackage{amsmath}
\usepackage{amsthm}
\usepackage{amsfonts}
\usepackage{amssymb}
\usepackage{adjustbox}
\usepackage{graphicx}
\usepackage{tikz}
\usetikzlibrary{backgrounds}
\usetikzlibrary{patterns,snakes}
\usepackage{mathtools}
\usepackage{enumitem}
\usepackage{dirtytalk}
\usepackage{wasysym}
\usepackage{fancybox}
\usepackage{float}
\usepackage{xfrac}
\usepackage{caption}
\usepackage{refcount}
\usepackage{centernot}

%% file: macros/macros.tex
\usepackage{hyperref}

\definecolor{lightgra}{rgb}{0.8, 0.8, 1.0}
\colorlet{lightgray}{lightgra!75}

\theoremstyle{plain}
\newtheorem{theorem}{Theorem}
\newtheorem{proposition}[theorem]{Proposition}
\newtheorem{definition}[theorem]{Definition}
\newtheorem{notation}[theorem]{Notation}
\newtheorem{lemma}[theorem]{Lemma}
\newtheorem{corollary}[theorem]{Corollary}
\newtheorem{question}{Question}

\newtheorem{claim}[theorem]{Claim}
\newtheorem{theorembis}{Theorem}

\theoremstyle{remark}

\newenvironment{claimproof}[1]{\par\addvspace{\baselineskip}\par\noindent\textit{Proof of the claim.}\space#1}{\hfill {$\square\ $Claim} \par\addvspace{\baselineskip}}

\newcommand{\bsigma}{\mathbf{\Sigma}}
\newcommand{\bpi}{\mathbf{\Pi}}
\newcommand{\bdelta}{\mathbf{\Delta}}
\newcommand{\bb}{\mathbf{\mathcal{B}}}
\newcommand{\A}{\mathcal{A}}

\newcommand{\PP}{\mathcal{P}}
\newcommand{\Pw}{\mathcal{P}\omega}
\newcommand{\scott}{\mathcal{P}\omega}
\newcommand{\infinite}{\mathcal{P}_{\w}(\omega)}
\newcommand{\finite}{\mathcal{P}_{<\w}(\omega)}
\newcommand{\B}{\mathcal{B}}
\newcommand{\w}{\omega}
\newcommand{\C}{\mathcal{C}}
\newcommand{\p}{\mathsf{P}}
\newcommand{\q}{\mathsf{Q}}
\newcommand{\F}{\mathcal{F}}
\newcommand{\N}{\mathbb{N}}

\newcommand{\OO}{\mathcal{O}}
\newcommand{\injectivehomo}{\xrightarrow[]{\text{\hbox{\resizebox{0.5cm}{!}{\textit{1-1 h.}}}}}}
\newcommand{\notinjectivehomo}{\centernot{\xrightarrow[]{\text{\hbox{\resizebox{0.5cm}{!}{\textit{1-1 h.}}}}}}}

\DeclareMathOperator{\Card}{Card}
\DeclareMathOperator{\Pred}{Pred}
\DeclareMathOperator{\Succ}{Succ}
\DeclareMathOperator{\rk}{rk}
\DeclareMathOperator{\lh}{lh}

\DeclareMathOperator{\reg}{emb}
\DeclareMathOperator{\lay}{emb}
\DeclareMathOperator{\Str}{Str}

\DeclareMathOperator{\incr}{incr}
\DeclareMathOperator{\imm}{im}
\DeclareMathOperator{\decr}{decr}

\DeclareMathOperator{\fin}{fin}
\DeclareMathOperator{\even}{even}
\DeclareMathOperator{\odd}{odd}
\DeclareMathOperator{\I}{I}
\DeclareMathOperator{\II}{II}

\DeclareMathOperator{\col}{c}

\DeclareMathOperator{\shrub}{shr}

%% file: macros/title.tex
\usepackage[auth-sc]{authblk}

\title{\textbf{The Wadge order on the Scott domain is not a well-quasi-order}}
\author[1]{Jacques Duparc} 
\author[1,2]{Louis Vuilleumier\footnote{The second author gratefully acknowledges support from the Swiss National Science Foundation grant 200021-159241.}} 
\affil[1]{Department of Information Systems (DESI),\authorcr University of Lausanne, 
Switzerland \authorcr \texttt{Jacques.Duparc@unil.ch}, \texttt{Louis.Vuilleumier.1@unil.ch} \authorcr ~} 
\affil[2]{Research Institute on the Foundations of Computer Science (IRIF), \authorcr Paris Diderot University, Sorbonne Paris Cité, %\authorcr Paris Cedex 13, 
France \authorcr \texttt{Louis.Vuilleumier@etu.univ-paris-diderot.fr}}
\date{}  